\documentclass[11pt]{article}

\usepackage{latexsym,color,amsmath,amsthm,amssymb,amscd,amsfonts,arydshln,cancel,setspace}
\usepackage{bbm}

\usepackage[left=2.0cm, top=2.45cm,bottom=2.45cm,right=2.0cm]{geometry}

\usepackage{tocloft}
\usepackage[T1]{fontenc}
\usepackage{lmodern}
\usepackage[utf8]{inputenc}

\usepackage{amsmath}
\usepackage{amssymb,amsthm,graphicx,xcolor,tikz}

\usepackage[british]{babel}
\usepackage{bbm}

\usepackage[
    bookmarks=true,
    bookmarksnumbered=true,
    bookmarksopen=true,
    unicode=true,
    pdftoolbar=true,
    pdfmenubar=true,
    pdffitwindow=false,
    pdfstartview={FitH},
    pdftitle={},
    pdfauthor={},
    pdfsubject={},
    pdfcreator={},
    pdfproducer={},
    pdfkeywords={},
    pdfnewwindow=true,
    colorlinks=true,
    linkcolor=black,
    citecolor=black,
    filecolor=black,
    urlcolor=black]{hyperref}
    
\makeatletter
    \def\@cite#1#2{[\textbf{#1}\if@tempswa , #2\fi]}	
    \def\@biblabel#1{[#1]}								
\makeatother

\numberwithin{equation}{section}
\numberwithin{figure}{section}

\newtheorem {theorem}{Theorem}[section]
\newtheorem {proposition}[theorem]{Proposition}
\newtheorem {lemma}[theorem]{Lemma}
\newtheorem {corollary}[theorem]{Corollary}
\newtheorem {example}[theorem]{Example}

\theoremstyle{definition}

\newtheorem {remark}[theorem]{Remark}

\newcommand{\R}{\mathbb{R}}
\renewcommand{\S}{\mathbb{S}}

\newcommand{\cA}{\mathcal{A}}

\newcommand{\cE}{\mathcal{E}}

\newcommand{\cC}{\mathcal{C}}

\newcommand{\bs}{\mathbf{s}}

\newcommand{\by}{\mathbf{y}}
\newcommand{\bx}{\mathbf{x}}

\newcommand{\bz}{\mathbf{z}}
\newcommand{\bu}{\mathbf{u}}
\newcommand{\bn}{\mathbf{n}}
\newcommand{\be}{\mathbf{e}}
\newcommand{\bv}{\mathbf{v}}
\newcommand{\bw}{\mathbf{w}}
\newcommand{\bo}{\mathbf{o}}

\DeclareMathOperator{\Vol}{vol}
\DeclareMathOperator{\vol}{vol}

\DeclareMathOperator{\interior}{int}

\DeclareMathOperator{\bd}{bd}

\DeclareMathOperator{\as}{as}
\DeclareMathOperator{\Sp}{Sp}

\DeclareMathOperator*{\argmin}{arg\,min}

\newcommand{\dint}{\mathrm{d}}
   
\renewcommand{\phi}{\varphi}

\begin{document}

\title{\bfseries The $L_p$-Floating Area, Curvature Entropy, and\\ Isoperimetric Inequalities on the Sphere}

\author{%
    Florian Besau\footnotemark[1]%
    \and Elisabeth~M.~Werner\footnotemark[2] %
\footnotemark[3]
}

\date{}
\renewcommand{\thefootnote}{\fnsymbol{footnote}}
\footnotetext[1]{%
    Technische Universität Wien, Austria. Email: florian@besau.xyz
}

\footnotetext[2]{%
    Case Western Reserve University, USA. Email: elisabeth.werner@case.edu}

\footnotetext[3]{Partially supported by  NSF grant DMS-2103482}

\maketitle

\begin{abstract}\noindent

\noindent
We explore analogs of classical centro-affine invariant isoperimetric inequalities, such as the Blaschke--Santaló inequality and the $L_p$-affine isoperimetric inequalities, for convex bodies in spherical space. Specifically, we establish an isoperimetric inequality for the floating area and prove a stability result based on the spherical volume difference.

The floating area has previously been studied as a natural extension of classical affine surface area to non-Euclidean convex bodies in spaces of constant curvature. In this work, we introduce the $L_p$-floating areas for spherical convex bodies, extending Lutwak’s centro-affine invariant family of $L_p$-affine surface area measures from Euclidean geometry. We prove a duality formula, monotonicity properties, and isoperimetric inequalities associated with this new family of curvature measures for spherical convex bodies.

Additionally, we propose a novel curvature entropy functional for spherical convex bodies, based on the $L_p$-floating area, and establish a corresponding dual isoperimetric inequality.

Finally, we extend our spherical notions to space forms with non-negative constant curvature in two distinct ways. One extension asymptotically connects with centro-affine geometry on convex bodies as curvature approaches zero, while the other converges with Euclidean geometry. Notably, our newly introduced curvature entropy for spherical convex bodies emerges as a natural counterpart to both the centro-affine entropy and the Gaussian entropy of convex bodies in Euclidean space.

        
\end{abstract}

\tableofcontents

\section{Introduction and main results}

The $L_p$-Brunn--Minkowski theory, an extension of the classical Brunn--Minkowski theory, lies at the foundation of geometric analysis. At its core are affine invariants associated with convex bodies, among which the \emph{$L_p$-affine surface areas} are most notable. These invariants, a family of centro-affine curvature measures, were introduced by Lutwak \cite{Lutwak:1996} and are defined as follows:
\begin{equation*}
    \as_p(K) = d \int_{\bd K} \kappa_0(K,\bx)^{\frac{p}{d+p}} \, V_K(\dint\bx) \qquad \text{for $p\geq 1$},
\end{equation*}
where 
$K$ is a compact convex set in $\R^d$ with the origin in its interior. In this formula, $\kappa_0(K,\bx)$ denotes the centro-affine curvature of 
$K$, and $V_K$ represents the cone volume measure; see Section \ref{sec:entropy} for precise definitions. 
This family of measures includes Blaschke's classical affine surface area when $p=1$ and the $0$-homogeneous centro-affine curvature measure when 
$p=d$. The latter, among other applications, arises naturally in Funk geometry, as discussed by Faifman \cite{Faifman:2024}.
\smallskip

Let $B_2^d$ be the centered Euclidean unit ball in $\R^d$ and let $B_K$ denote the centered Euclidean ball such that $\vol_d(B_K)=\vol_d(K)$. The \emph{$L_p$-affine isoperimetric inequality}, for a centered convex body $K\subset \R^d$, is
\begin{equation}\label{eqn:Lp-affine-ineq}
    \as_p(K) \leq \as_p(B_K) = d\vol_d(B_2^d)^{\frac{2p}{d+p}} \vol_d(K)^{\frac{d-p}{d+p}}, 
\end{equation}
with equality if $K$ is a ellipsoid centered at the origin.
These measures have proven to be invaluable in various contexts, particularly as they represent the natural (centro-)affine invariant curvature measures within the theory of valuations on convex bodies \cite{LR:1999, LR:2010, HP:2014}; see also \cite{Alesker:2018, BKW:2023, Knoerr:2023, LM:2023, SW:2018}. The inequalities \eqref{eqn:Lp-affine-ineq} have been further generalized to cases where $p\in\R\setminus\{-d\}$ and extended to Orlicz-affine surface areas, as explored in \cite{Ludwig:2010, LR:2010, MW:2000, Werner:2002, WY:2008, Ye:2015}. Additionally, dual Orlicz-affine isoperimetric inequalities were developed by Ye \cite{Ye:2016}. For insights into a centro-projective isoperimetric inequality, see Vernicos \& Yang \cite{VY:2019}, and for the $L_p$-Steiner formula, see \cite{TW:2019, TW:2023}.
\smallskip

Recently, several extensions of concepts from Euclidean Brunn--Minkowski theory to non-Euclidean geometries have been developed; see \cite{BHRS:2017, BGRSTW:2024,BHPS:2023, DKY:2018, KMTT:2019}. In our earlier work, we introduced a natural non-Euclidean counterpart to Blaschke's affine surface area, known as the \emph{floating area} \cite{BLW:2018, BW:2016, BW:2018}, for convex bodies in spaces of constant curvature. In our most recent work \cite{BW:2024}, we proposed a natural analog to the $L_p$-affine surface area for $p=-d/(d+2)$. In this paper, we aim to extend this family to all $p\in\R\setminus\{-d\}$ and explore new isoperimetric inequalities.
\smallskip

Closely related to inequality \eqref{eqn:Lp-affine-ineq} is the classical \emph{Blaschke--Santaló inequality}, see, for example, the recent survey \cite{FMZ:2023}. For a convex body $K\subset \R^d$ with non-empty interior, it states that 
\begin{equation}\label{eqn:BS}
    \min_{\bx\in K} \vol_d((K-\bx)^\circ) \leq \vol_d(B_K^\circ)= \vol_d(B_2^d)^2 \vol_d(K)^{-1},
\end{equation}
where $K^\circ=\{\by : \bx\cdot \by \leq 1 \text{ for all $\bx\in K$}\}$ denotes the polar body of $K$.
This inequality is invariant under affine transformations, and equality in \eqref{eqn:BS} holds if and only if $K$ is an ellipsoid.

The $L_p$-affine isoperimetric inequality in \eqref{eqn:Lp-affine-ineq} can be derived from the Blaschke--Santaló inequality, as for a centered convex body $K$, we have \begin{equation}\label{eqn:Lp-affine-weaker}
    \as_p(K) \leq d\vol_d(K^\circ)^{\frac{p}{d+p}} \vol_d(K)^{\frac{d}{d+p}}.
\end{equation}
Notably, the Blaschke--Santaló inequality and the $L_p$-affine isoperimetric inequality for $p=1$ are equivalent.
\smallskip

It is well-known that the minimal position of $K$ in the left-hand side of \eqref{eqn:BS} is realized by a uniquely determined point $\bs=\bs(K)$ in the interior of $K$--the \emph{Santaló point}. Then 
\begin{equation*}
    \vol_d((K-\bs)^\circ) = \min_{\bx\in K} \vol_d((K-\bx)^\circ).
\end{equation*}
Moreover, $\bx\in K$ is the Santaló point if and only if $(K-\bx)^\circ$ is centered, that is, $(K-\bx)^\circ$ has its centroid at the origin. Thus, if $K$ is centered and since $K=(K^\circ)^\circ$, \eqref{eqn:BS} is equivalent to
\begin{equation*}
    \vol_d(K^\circ) \leq \vol_d(B_K^\circ).
\end{equation*}
\smallskip

Stability results for inequality \eqref{eqn:Lp-affine-ineq} were established in \cite{CW:2017}, while stability versions for the Blaschke--Santaló inequality \eqref{eqn:BS} were first studied by Böröczky \cite{Bor:2010}; see also \cite{BH:2010, Ivaki:2014, Ivaki:2015}. Over the past decade, numerous generalizations of \eqref{eqn:BS} have emerged, including functional analogs \cite{AAKM:2004, CKLR:2024, FGSZ:2024, FM:2007, HJM:2020, KW:2021, Rotem:2014}, with stability versions in \cite{BBF:2014}.

Further extensions include $L_p$ \cite{LYZ:2000, LZ:1997} and Orlicz versions \cite{XZ:2022}, which have also led to a new sine Blaschke--Santaló inequality for the sine-polar body of star bodies \cite{HLXY:2022}; lower-dimensional projections \cite{Vritsiou:2024}; the extremal affine surface area \cite{GHSW:2020, Hoehner:2022}; and higher-order convex bodies and functions \cite{HLPRY:2023, LMU:2024}.

Related inequalities for difference bodies, projection bodies, and more general Minkowski endomorphisms can be found in \cite{HS:2019, HS:2023, LRZ:2023, Lutwak:1996:2}, while versions for complex projection bodies appear in \cite{AB:2011, Haberl:2019}. Recently, E.\ Milman \& Yehudayoff \cite{MY:2023} resolved a major conjecture by Lutwak by establishing sharp isoperimetric inequalities for the entire family of affine quermassintegrals, which includes the Blaschke--Santaló inequality; see also \cite{DPP:2019} for a generalization to (dual) affine quermassintegrals for flag-manifolds.
\smallskip

Thus, extending the Blaschke--Santaló inequality to non-Euclidean geometries is of particular interest. In spherical space, this extension was first investigated by Gao, Hug, \& Schneider \cite{GHS:2003}. It gives
\begin{equation}\label{eqn:BS_spherical}
    \vol_d^s(K^*) \leq \vol_d^s(C_K^*),
\end{equation}
for a spherical convex body $K\subset \S^d$, where $K^*$ is the spherical dual body and $C_K$ is a geodesic ball such that $\vol_d^s(K)=\vol_d^s(C_K)$, see also \cite{HP:2021} for a randomized version.

More recently, Hug \& Reichenbacher \cite{HR:2017} derived a stability version for this spherical analog of the Blaschke--Santaló inequality, also see Theorem \ref{thm:stab_HR}, which states that
\begin{equation}\label{eqn:stab_HR}
    \vol_d^s(K^*) \leq \left(1- C \Delta(K)^2\right) \vol_d^s(C_K^*),
\end{equation}
where $C=C(K)>0$ and $\Delta(K) = \vol_d^s(K\triangle C_K(\bo))$ measures the distance between $K$ and the geodesic ball of the same volume centered at a special point $\bo=\bo(K)\in\interior K$, which we call the \emph{GHS-center} of $K$. Recently, also the isodiametric inequality with stability was proven for spherical convex bodies by Böröczky \& Sagmeister \cite{BS:2020,BS:2023}, also see \cite{HCMF:2015} for an earlier stronger result in the two dimensional case.
\smallskip

The spherical analog of the Blaschke--Santaló inequality \eqref{eqn:BS_spherical} differs remarkably from its classical Euclidean counterpart \eqref{eqn:BS}. 
In the Euclidean case, the inequality depends on the centro-affine geometry of $K$ and thus on the position of the origin. In contrast, the spherical version is invariant under any isometry of spherical space, meaning there is no naturally distinguished point within $K\subset\S^d$ as there is in the Euclidean case. However, from the stability version \eqref{eqn:stab_HR}, we see that the GHS-center provides a natural point within $K$ that may play a role similar to the Euclidean centroid. It should be noted, however, that the GHS-center generally differs from other classical notions of spherical centroids of $K$; see Remark \ref{rmk:GHS}.

\subsection{\texorpdfstring{$L_p$}{Lp}-floating area and isoperimetric inequalities on the sphere}

As previously mentioned, the \emph{floating area} was introduced in earlier work \cite{BW:2016, BW:2018, BW:2024} as a natural analog to Blaschke's affine surface area. In this paper, we build on these investigations, focusing specifically on isoperimetric inequalities in spherical space.
For $d\geq 2$, $p\in \R\setminus\{-d\}$, and a proper spherical convex body $K\subset \S^d$ (where we call $K$ proper if it is contained in an open half-space), we introduce the \emph{$L_p$-floating area}
\begin{equation*}
    \Omega_p^s(K) := \int_{\bd K} H_{d-1}^s(K,\bu)^{\frac{p}{d+p}}\, \vol_{\bd K}^s(\dint\bu).
\end{equation*}
Here, $H_{d-1}^s(K,\cdot)$ denotes the (generalized) spherical Gauss--Kronecker curvature and $\vol_{\bd K}^s$ is the $(d-1)$-dimensional Hausdorff measure of $\bd K\subset \S^d$. For $p=1$ this gives the floating area, which was naturally derived in \cite{BW:2016,BW:2018} as the volume-derivative of the floating body. For $p=-d(d+2)$ this $L_p$-floating area was derived in \cite{BW:2024} as the volume-derivative of the dual body of the floating body.

For $p=1$ an analog to \eqref{eqn:Lp-affine-weaker} was derived in \cite[Thm.\ 7.6]{BW:2016}, and in Theorem \ref{thm:sphere_inequalities}, we establish that
\begin{equation}\label{eqn:Lp-sphere-weaker}
    \Omega_p^s(K) \leq P^s(K^*)^{\frac{p}{d+p}} P^s(K)^{\frac{d}{d+p}},
\end{equation}
for $p>0$, where $P^s(K)=\vol_{\bd K}^s(\bd K)$ is the spherical surface area (or perimeter) of $K$. 
Equality holds in \eqref{eqn:Lp-sphere-weaker} if and only if $K$ is a geodesic ball.
\smallskip

Compactness of $\S^d$ and the upper semi-continuity of $\Omega_p^s$ leads us to conjecture that
\begin{equation}\label{eqn:floating_ineq}
    \Omega_p^s(K) \overset{?}\leq P^s(C_K^*)^{\frac{p}{d+p}} P^s(C_K)^{\frac{d}{d+p}}= \Omega_p^s(C_K) \qquad \text{for $p\geq 1$}.
\end{equation}

In the Euclidean case, for $p=1$, we have seen that \eqref{eqn:Lp-affine-weaker} gives rise to the affine isoperimetric inequality \eqref{eqn:Lp-affine-ineq} by the Blaschke--Santaló inequality. Remarkably, \eqref{eqn:Lp-sphere-weaker} does not directly compare to \eqref{eqn:floating_ineq} in the same way.
Indeed, we have by the classical isoperimetric inequality 
\begin{equation*}
    P^s(C_K)\leq P^s(K).
\end{equation*}
Moreover, recent major progress on the conjecture Aleksandrov--Fenchel inequalities in spherical space were made using curvature flow techniques, see \cite{CGLS:2022,CS:2022, HL:2023, MS:2016, WX:2015}, which, for $d\geq 3$, give rise to the dual isoperimetric inequality
\begin{equation}\label{eqn:dual_iso}
    P^s(K^*) \leq P^s(C_K^*),
\end{equation}
for spherical convex bodies of class $\cC^2_+$, see Section \ref{sec:curvature_flow} where we review the current state of the art for (dual) isoperimetric inequalities on the sphere. 
Thus for $d\geq 3$ the upper bounds in \eqref{eqn:Lp-sphere-weaker} and \eqref{eqn:floating_ineq} can not be compared in general.
However, for $d=2$ we have $P^s(K)=P^s(C_K^*)$ for all proper spherical convex bodies $K$, which yields that the upper bound in \eqref{eqn:floating_ineq} is stronger than the bound in \eqref{eqn:Lp-sphere-weaker}.
Progress has recently been made by Basit, Hoehner, Lángi \& Ledford \cite{BHLL:2024}, who were able to verify that for $p=1$ inequality \eqref{eqn:floating_ineq} holds true for smooth symmetric spherical convex bodies on $\S^2$, which therefore improves \eqref{eqn:Lp-affine-weaker} in this case.

\medskip
We use the notion of GHS-center, see Section \ref{sec:GHS-center}, to derive the following
\begin{theorem}
    Let $d\geq 3$ and let $K\subset \S^d$ be a spherical convex body of class $\cC^2_+$ such that either $K\subset C_{\bo}(\alpha_d)$ or $K\supset C_{\bo}(\beta_d)$, where $\tan \alpha_d=\sqrt{d(d-2)}$, $\tan\beta_d = \sqrt{d}$, and $\bo=\bo(K)$ is the GHS-center of $K$. Then
    \begin{equation}\label{eqn:lp_float}
        \Omega^s_p(K)\leq \Omega^s_p(C_K) \qquad \text{for $p\geq 1$},
    \end{equation}
    with equality if and only if $K$ is geodesic ball.
\end{theorem}
This result follows in Corollary \ref{cor:p-float_ineq}, by first establishing the result for $\Omega_1^s$ (without the $\cC^2_+$ curvature condition), see Theorem \ref{thm:sfa-ineq}, and combining it with
the duality
\begin{equation*}
    \Omega_p^s(K^*) = \Omega_{d^2/p}^s(K),
\end{equation*}
see Theorem \ref{thm:duality}, as well as a monotonicity result for $\Omega_p^s$, see Theorem \ref{thm:sphere_inequalities}, and the dual isoperimetric inequality \eqref{eqn:dual_iso} which we recall in Proposition \ref{cor:polar_iso}. We firmly believe that the containment conditions $K\subset C_{\bo}(\alpha_d)$ or $K\supset C_{\bo}(\beta_d)$ are superfluous and \eqref{eqn:lp_float} should hold true for any proper spherical convex body for $d\geq 2$. However, our current strategy of proof apparently cannot be employed without these additional assumptions. 

\smallskip
Moreover, for $p=1$, we also derive the following stability result.
\begin{theorem}
    Let $d\geq 3$ and let $K\subset \S^d$ be a spherical convex body such that $K\subset C_{\bo}(\alpha_d)$. Then
    \begin{equation*}
        \Omega^s_1(K) \leq (1-C\Delta(K)^2)\, \Omega^s_1(C_K),
    \end{equation*}
    where $C=C(K)>0$, and $\Delta(K)=\vol_d^s(K\triangle C_K(\bo))$ measures the distance between $K$ and the geodesic ball centered in the GHS-center of $K$ with the same volume as $K$.
\end{theorem}
This result relies on tools developed by Hug \& Reichenbacher \cite{HR:2017} and will be proven in Section \ref{sec:stab}, see Theorem \ref{thm:stability_fa}.

\smallskip
Finally, we remark that for $p\to 0^+$ the $L_p$-floating area reduces to the perimeter functional $P^s(K)=\Omega_0^s(K)$ and \eqref{eqn:lp_float} is then reversed by the classical isoperimetric inequality. The behavior of $\Omega_p^s$ for $p\in(0,1)$ therefore seems more complex, but in Corollary \ref{cor:p-float_ineq} we do also derive a version of \eqref{eqn:lp_float} for $p\in(0,1)$ under certain containment conditions that then necessarily depend on $p$.

\subsection{Spherical curvature entropy and real-analytic extensions to space forms}

The $L_p$-affine surface was used by Paouris \& Werner in \cite{PW:2012} to introduce a centro-affine invariant \emph{entropy power functional} $\cE$. For $K\subset \R^d$ with the origin in the interior, they define
\begin{equation}\label{eqn:PW_power}
    \cE(K) = \lim_{q\to 0^+} \left(\frac{\as_q(K^\circ)}{d\vol_d(K^\circ)}\right)^{1+\frac{d}{q}} \in [0,+\infty),
\end{equation}
where $K^\circ$ denotes the polar body of $K$.

Note that $-\log\cE(K)$ is related to the Boltzmann entropy of the centro-affine curvature of $K^\circ$ with respect to the cone volume measure $V_{K^\circ}$, that is,
\begin{equation*}
    E_{PW}(K):= -\log \cE(K) = -\frac{1}{\vol_d(K^\circ)} \int_{\bd K^\circ} \left[\log \kappa_0(K^\circ,\bx)\right] \, V_{K^\circ}(\dint \bx).
\end{equation*}
Moreover, if $K$ is of class $\cC^{1,1}$, then
\begin{equation*}
    E_{PW}(K) = \frac{1}{\vol_d(K^\circ)} \int_{\bd K} \kappa_0(K,\bx) \left[\log \kappa_0(K,\bx)\right] \, V_K(\dint\bx).
\end{equation*}
The polar entropy inequality
\begin{equation*}
    \cE(K^\circ) \leq \left(\frac{\vol_d(B_2^d)}{\vol_d(K^\circ)}\right)^2= \cE(B_K^\circ),
\end{equation*}
was established in \cite[Prop.\ 3.10]{PW:2012} for a centered convex body $K\subset \R^d$ of class $\cC^2_+$, with equality if and only if $K$ is an ellipsoid.
For $d\geq 2$, it is apparently still an open conjecture that
\begin{equation}\label{eqn:PW_conj}
    \cE(K)\overset{?}{\leq }\left(\frac{\vol_d(K)}{\vol_d(B_2^d)}\right)^2= \cE(B_K),
\end{equation}
for a centered convex body $K\subset \R^d$. By the Blaschke--Santaló inequality, the polar entropy inequality would follow from \eqref{eqn:PW_conj}. In particular, \eqref{eqn:PW_conj} would imply that the centro-affine entropy $E_{PW}$ is non-negative for convex bodies $K$ with the same volume as $B_2^d$, i.e.,
\begin{equation*}
    E_{PW}(K) = -\log \cE(K) \overset{?}{\geq} -2 \log \frac{\vol_d(K)}{\vol_d(B_2^d)} = E_{PW}(B_K). 
\end{equation*}
\smallskip

Another entropy functional was used by Gage \& Hamilton \cite{GH:1986} in the context of geometric flows for the boundary curve of a $2$-dimensional convex body and extended to $d$-dimensions by Chow \cite{Chow:1991}. Following Chow \cite{Chow:1991}, we define the \emph{Gaussian entropy}
\begin{equation*}
    E_C(K) = \frac{1}{\omega_{d-1}} \int_{\bd K} H_{d-1}(K,\bx) \left[\log H_{d-1}(K,\bx)\right] \, \dint \bx,
\end{equation*}
for a convex body $K\subset \R^d$ of class $\cC^2_+$ and $\omega_{d-1}=P(B_2^d)$ is the total surface area of the unit sphere $\S^{d-1}=\bd B_2^d$. Using the affine isoperimetric inequality we derive the well-known lower bound
\begin{equation*}
    E_C(K) \geq -\frac{d+1}{d} \log \frac{\as_1(K)}{\omega_{d-1}} \geq -\frac{d-1}{d} \log \frac{\vol_d(K)}{\vol_d(B_2^d)} = E_C(B_K).
\end{equation*}
This was further improved by Guan \& Ni \cite{GN:2017} who related $E_C$ to an extension of another entropy functional that was originally introduced by Firey \cite{Fiery:1977} and previously employed by Andrews \cite{Andrews:1997, Andrews:1999}. We refer to Section \ref{sec:entropy} where we briefly recall these different entropy functionals on Euclidean convex bodies and their relations.
\smallskip

One may consider entropy heuristically as a measure of ``complexity''. Specifically, the Gaussian entropy is minimized for balls and tends towards infinity as a convex body approaches a polytope. This perspective has been explored in \cite{OMK:2020,UKSM:2012}, where Gaussian entropy of curves and surfaces is used as an index for shape analysis and design.
\smallskip

Following \eqref{eqn:PW_power} we introduce the \emph{spherical entropy power functional} by
\begin{equation*}
    \cE^s(K) := \lim_{q\to 0^+} \left(\frac{\Omega_q^s(K^*)}{P^s(K^*)}\right)^{1+\frac{d}{q}},
\end{equation*}
for a proper spherical convex body $K\subset \S^d$. This gives rise to the following curvature entropy functional.

\begin{theorem}[Spherical Curvature Entropy]
    Let $d\geq 2$ and let $K\subset \S^{d}$ be a proper spherical convex body. Then
    \begin{equation*}
        E^s(K) := - \log \cE^s(K) = \frac{1}{P^s(K^*)} \int_{\bd K^*} \log H^s_{d-1}(K^*,\bu) \, \vol_{\bd K^*}^s(\dint \bu),
    \end{equation*}
    is called \emph{spherical curvature entropy} of $K$.
    
    Furthermore, if $K$ is of class $\cC^{1,1}$, then
    \begin{equation}\label{eqn:entropy_formula_sphere}
        E^s(K) = \frac{1}{P^s(K^*)} \int_{\bd K} H_{d-1}^s(K,\bu) \left[\log H_{d-1}^s(K,\bu)\right] \, \vol_{\bd K}^s(\dint \bu).
    \end{equation}
\end{theorem}
We proof this Theorem in Section \ref{sec:entropy_sphere}. We note here that by \eqref{eqn:entropy_formula_sphere} the spherical curvature entropy $E^s$ can also be considered as a spherical analog to the Gaussian entropy in spherical space. Indeed, in Section \ref{sec:real_limit} we investigate two different real-analytic extensions to $d$-dimensional space forms $\Sp^d(\lambda)$ of constant curvature $\lambda\geq 0$. The first extension arises by fixing a point in the space-forms and rescaling radially around the point. This will connect $E^s$ in the limit $\lambda\to 0^+$ with the centro-affine invariant entropy $E_{PW}$. The second extension does not fix a point and just rescales by the curvature uniformly. In this way $E^s$ will connect in the limit $\lambda\to 0^+$ with $E_C$. Thus $E^s$ can be seen as spherical analog to both, the centro-affine entropy $E_{PW}$ and the Gaussian entropy $E_C$. A similar observation for $\Omega^s_p$ is drawn, which builds on previous investigations in \cite{BW:2024} for the special case $p=-d/(d+2)$.
\smallskip

We establish the following isoperimetric inequalities for the spherical entropy power.
\begin{theorem}
    For $d\geq 2$ and a proper spherical convex body $K\subset \S^d$, we have the \emph{spherical information inequality}
    \begin{equation*}
        \cE^s(K) \leq \frac{P^s(K)}{P^s(K^*)},
    \end{equation*}
    with equality if and only if $K$ is a geodesic ball.
    
    Moreover, we have the \emph{dual entropy inequality}
    \begin{equation*}
        \cE^s(K^*) \leq \frac{P^s(C_K^*)}{P^s(C_K)} = \cE^s(C_K^*),
    \end{equation*}
    where we additionally assume that $K$ is class $\cC^2_+$ for $d\geq 3$. Equality holds again if and only if $K$ is a geodesic ball.
\end{theorem}
These result follows in Corollary \ref{cor:entropy_bound} and Theorem \ref{thm:entropy_polar}, where for the dual entropy inequality for $d\geq 3$ we again apply the dual isoperimetric inequality \eqref{eqn:dual_iso}, which is the reason why we have to impose the additional curvature condition on $K$.
\smallskip

Finally, we mention that we conjecture that for $d\geq 2$ and a proper spherical convex body $K$ we have
\begin{equation}\label{eqn:entropy_sphere_conj}
    \cE^s(K) \overset{?}{\leq} \frac{P^s(C_K)}{P^s(C_K^*)} = \cE^s(C_K).
\end{equation}
This would also imply the dual entropy inequality by the spherical Blaschke--Santaló inequality, see Remark \ref{rmk:conjecture_entropy}.

For $d=2$ we can verify \eqref{eqn:entropy_sphere_conj} by connecting it to \eqref{eqn:floating_ineq}. Thus, from the results of Basit, Hoehner, L\'angi \& Ledford \cite{BHLL:2024}, we derive in Corollary \ref{cor:entropy_strong_bound} that \eqref{eqn:entropy_sphere_conj} is true for symmetric convex bodies $K\subset \S^2$ of class $\cC^2_+$. Thus we conclude the following corollary in terms of our spherical curvature entropy on $\S^2$.

\begin{corollary}[Positivity of spherical curvature entropy on $\S^2$] Let $K\subset \S^2$ be a symmetric convex body of class $\cC^2_+$. Then
    \begin{equation}\label{eqn:positivity}
        E^s(K) \geq E^s(C_K) 
        = - \frac{1}{2} \log \frac{4\pi \vol_2^s(K)-\vol_2^s(K)^2}{(2\pi - \vol_2^s(K))^2}.
    \end{equation}
    In particular, $E^s(K)\geq 0$ if $\vol_2^s(K) \leq \vol_2^s(C(\pi/4))=(2-\sqrt{2})\pi$, or, by the isodiametric inequality \cite{BS:2020,HCMF:2015}, if $\operatorname{diam} K\leq \frac{\pi}{2}$.
\end{corollary}
\medskip

\noindent
\textbf{Organization of the paper.} In Section 2 we recall basic notions on spherical convex bodies and recall the definition of, what we call, GHS-centers. In Section 3 we review (dual) isoperimetric inequalities for spherical convex bodies that follow from recent advances on the analog to the Aleksandrov--Fenchel inequalities in spherical space. In Section 4 and 5 we establish the isoperimetric inequality for the floating area $\Omega_1^s$ and a stability result. In Section 6 we introduce the family of $L_p$-floating areas and establish a duality formula, monotonicity and isoperimetric inequalities. In Section 7 we recall the different entropy functionals in Euclidean space, before introducing our spherical curvature entropy in Section 8. Finally, in Section 9 we discuss two real analytic extensions to space forms of constant curvature $\lambda\geq 0$ which have properties similar to the $L_p$-floating area and the spherical curvature entropy. In particular, these families connect $\Omega_p^s$, respectively $E^s$, for $\lambda\to 0^+$ in the Euclidean space, either with the $L_p$-affine surface area or a Euclidean curvature measure, respectively, with either the centro-affine entropy $E_{PW}$ or the Gaussian entropy $E_C$.

\section{Special centers of spherical convex bodies}\label{sec:GHS-center}

We denote by $\S^d:=\{\bu\in\R^{d+1}:\bu\cdot\bu=1\}$ the unit sphere in $\R^{d+1}$ where $\cdot$ is the Euclidean scalar product and we write $\|\cdot\|_2$ for the Euclidean norm. The total measure of $\S^d$ is $\omega_d:=\vol_d^s(\S^d) = 2 \pi^{\frac{d+1}{2}} / \Gamma(\frac{d+1}{2})$ and $\kappa_d:=\vol_d^e(B_2^d) = \omega_{d-1}/d = \pi^{\frac{d}{2}} / \Gamma(\frac{d}{2}+1)$. The geodesic distance $d_s(\bu,\bv)$ for $\bu,\bv\in \S^d$ is determined by $\cos d_s(\bu,\bv) =\bu\cdot\bv$.

A set $K\subset \S^d$ is a convex body, if $K$ is closed and for any two points $\bu,\bv\in K$ with $\bu\cdot \bv>0$ we have that the uniquely determined minimal geodesic  arc connecting $\bu$ and $\bv$ is also contained in $K$. In addition, $K$ is called \emph{proper} if $K$ is contained in an open half-sphere, that is, if there exist $\be\in\S^d$ such that $K\subset \interior H^+(\be) = \{\bv\in \S^d: \bv\cdot\be >0\}$.

The \emph{(spherical) dual body $K^*$} is defined by
\begin{equation*}
    K^*
    = \bigcap_{\bu\in K} \{\bv \in \S^d : \bv\cdot \bu\geq 0\}.
\end{equation*}
Notice that $\bv\in K^*$ if and only if $K\subset H^+(\bv)$ and $\bv\in \interior K^*$ if and only if $K\subset \interior H^+(\bv)$.
Thus $K$ is proper if and only if $K^*$ has non-empty interior.

We have $\bo \in \interior K$ and $K\subset \interior H^+(\bo)$ if and only if $\bo \in \interior K \cap \interior K^*$.
For any $\bo\in \interior K\cap \interior K^*$, $K$ is uniquely determined by the \emph{(spherical) radial function} $\rho_{\bo}(K,\cdot): \S^d\cap\bo^\bot\cong \S^{d-1} \to (0,\frac{\pi}{2})$, with respect to the center $\bo$, defined by
\begin{equation*}
    \rho_{\bo}(K,\bu) = \max \left\{\lambda\in(0,+\infty) : (\cos \lambda)\bo+(\sin\lambda)\bu\in K\right\} \qquad \text{for all $\bu\in\S^d\cap\bo^\bot$}.
\end{equation*}
By using radial coordinates on $\S^d$ in $\bo$, we can express the spherical volume of $K$ by
\begin{equation}
    \vol_d^s(K) = \int_{\S^{d}\cap\bo^\bot} J(\rho_{\bo}(K,\bu))\, \dint\bu,
\end{equation}
where $J(\alpha) := \int_{0}^{\alpha} (\sin t)^{d-1}\, \dint t$.

\begin{example}
        The geodesic ball $C_{\bo}(\alpha) = \{\bu\in\S^d : d_s(\bo,\bu)\leq \alpha\}$ has the dual
        \begin{equation*}
            C_{\bo}(\alpha)^* = C_{\bo}\left(\frac{\pi}{2}-\alpha\right).
        \end{equation*}
        In particular, the geodesic ball of radius $\frac{\pi}{4}$ is self-dual, i.e., $C_{\bo}(\frac{\pi}{4})^* = C_{\bo}(\frac{\pi}{4})$. 
        The spherical volume of $C_{\bo}(\alpha)$ is 
        \begin{equation}
            \vol_{d}^s(C_{\bo}(\alpha)) = \omega_{d-1} J(\alpha).
        \end{equation}
\end{example}

For $\bo\in\interior K \cap \interior K^*$ we can identify $K$ with a Euclidean convex body $\overline{K}\subset \R^{d+1}\cap\bo^\bot\cong\R^d$ in the projective model of $\S^d$ around $\bo$ via the gnomonic projection $g_{\bo} : \interior H^+(\bo)\to \R^{d+1}\cap\bo^\bot\cong\R^d$ defined by $g_{\bo}(\bu) = \frac{\bu}{\bu\cdot\bo}-\bo$. If $\bo=\be_{d+1}$, then the gnomonic projection $g=g_{\be_{d+1}}$ is given in coordinates by $g(\bu) = \frac{1}{u_{d+1}} (u_1,\dotsc,u_d)$. We refer to $\overline{K}=g_{\bo}(K)$ as \emph{projective model} (or \emph{affine chart}) of the spherical convex body $K$.

For $0\in\interior \overline{K}$, the Euclidean radial function $\rho_{\overline{K}}:\S^{d-1}\to (0,+\infty)$ of $\overline{K}$ is defined by $\rho_{\overline{K}}(\bu) = \max\{t>0:t\bx\in\overline{K}\}$. It is related to the spherical radial function by
\begin{equation*}
    \tan \rho_{\bo}(K,\bu) = \rho_{\overline{K}}(\bu),
\end{equation*}
see \cite[Lem.\ 7.2]{BS:2016}. Furthermore, the spherical dual body $K^*$ is related to the polar body $\overline{K}^\circ$ by (see \cite[Prop.\ 7.3]{BS:2016}), 
\begin{equation*}
    g_{\bo}(K^*) = -\overline{K}^\circ.
\end{equation*}

Gao, Hug \& Schneider observed that there is a uniquely determined affine chart $\overline{K}=g_{\bo}(K)$ of a proper spherical convex body $K$ such that $\overline{K}$ is centered, that is, $\overline{K}$ has the centroid at the origin. We include the short proof following the arguments in \cite{GHS:2003} for the reader's convenience.

\begin{theorem}[{Gao, Hug \& Schneider \cite[p.~171]{GHS:2003} -- GHS-center}]\label{thm:GHS}
    Let $K\subset \S^d$ be a proper spherical convex body with non-empty interior. Then there exists a uniquely determined point $\bo=\bo(K)\in \interior K\cap\interior K^*$, such that $\overline{K}:=g_{\bo}(K)$ has centroid at the origin, i.e.,
    \begin{equation}\label{eqn:cent}
        \int_{\overline{K}} (\bx\cdot\bu)\, \dint \bx = \int_{K} (\bw\cdot\bo)^{-(d+2)} (\bw\cdot\bu)\, \dint\bw = 0 \qquad \text{for all $\bu\in\S^{d}\cap\bo^\bot\cong \S^{d-1}$}.
    \end{equation}
    Moreover, $\bo^*=\bo(K^*)$ is the uniquely determined point in $\interior K\cap \interior K^*$ such that $\overline{K}=g_{\bo^*}(K)$ has Santaló point at the origin.
\end{theorem}
\begin{proof}
    We define $F_{K^*}:K^* \to (0,+\infty]$ by
    \begin{equation*}
        F_{K^*}(\bv) = \vol_d^e(g_{\bv}(K)) = \int_{K} [\cos d_s(\bv,\bw)]^{-(d+1)}\, \dint\bw.
    \end{equation*}
    For $\bv\in\bd K^*$ we have that $g_{\bv}(K)$ is unbounded and therefore $F_{K^*}(\bv)=+\infty$. Since $F_{K^*}$ is continuous and by compactness, there is a minimum $\bo\in \interior K^*$. Since $t\mapsto t^{-(d+1)}$ is strictly convex, the minimum $\bo$ is uniquely determined. Furthermore, the directional derivatives of $F_{K^*}$ at $\bo$ vanish, that is,
    \begin{align*}
        0 = \lim_{\varepsilon\to 0} \frac{F_{K^*}((\cos\varepsilon)\bo+(\sin\varepsilon)\bu)-F_{K^*}(\bo)}{\varepsilon} 
        &= (d+1) \int_{K} (\bw\cdot \bo)^{-(d+2)} (\bw\cdot \bu)\, \dint\bw \\
        &= (d+1) \int_{\overline{K}} (\bx\cdot \bu) \,\dint \bx,
    \end{align*}
    for all $\bu\in\S^{d}\cap \bo^\bot$.
    Thus $\overline{K}$ has the centroid at the origin, which implies that $g_{\bo}(\bo) = 0 \in \interior \overline{K} = \interior g_{\bo}(K)$, i.e., $\bo\in\interior K$.
 
 For the second statement we just note that $\overline{K}$ has the Santaló point at the origin if and only if $\overline{K}^\circ = -g_{\bo}(K^*)$ is centered, see, for example, \cite[Eq.~(10.23)]{Schneider:2014}.
\end{proof}

\begin{remark}[the GHS-center $\bo(K)$ is a $H$-barycenter of the uniform measure on $K$]\label{rmk:GHS}
    Let $K\subset \S^d$ be a proper spherical convex body. Arguing as in the proof of Theorem \ref{thm:GHS}, we observe the following:
    For $\alpha\geq 1$, the function $H_{\alpha}(t):=[\cos t]^{-\alpha}$, for $t\in[0,\frac{\pi}{2})$, is strictly convex and strictly increasing and there exists a uniquely determined point $\bo_\alpha(K)\in \interior K^*$ such that
    \begin{align*}
        \bo_\alpha(K) &= \argmin_{\bv\in \interior K^*} \int_{K} H_\alpha(d_s(\bv,\bw))\, \frac{\dint\bw}{\vol_d^s(K)}
        = \argmin_{\bv\in\interior K^*} \int_{g_{\bv}(K)} (1+\|\bx\|^2)^{\frac{\alpha - (d+1)}{2}}\, \frac{\vol_d^e(\dint\bx)}{\vol_d^s(K)}.
    \end{align*}
    We note that $\bo_\alpha(K)$ is the $H_\alpha$-barycenter of the uniform measure $\mu_K$ on $K$, i.e., $\mu_K(A):=\frac{\vol_d^s(K\cap A)}{\vol_d^s(K)}$, and uniqueness can also be determined by a general result on uniqueness of $H$-barycenters in $\mathrm{CAT(1)}$-spaces, see \cite[Thm.\ 41]{Yokota:2017}, although here we need to assume that there exists $\bw\in\S^d$ such that $K\subset C_{\bw}(\frac{\pi}{4})$. We note that there is some ambiguity about the terminology in the literature and instead of \emph{barycenter of a measure}, one may also call this notion a \emph{Fréchet mean}, \emph{Karcher mean} or \emph{(Riemannian) center or mass}.
\end{remark}

\section{Dual isoperimetric inequalities on the sphere} \label{sec:curvature_flow}

Recall that we  write $C_K(\bo)=C_{\bo}(\alpha_K)$ for the uniquely determined geodesic ball (spherical cap) in $\S^d$ center in $\bo$ with the same volume as $K$ and we just use $C_K=C_K(\bo)$ when the expression does not depend on the actual center $\bo$.  Furthermore, a convex body $K\subset \S^d$ is called of class $\cC^2_+$, if $\bd K$ is a $\cC^2$-smooth embedded hypersurface of $\S^d$ with everywhere strictly positive Gauss--Kronecker curvature $H_{d-1}^s(K,\cdot)$. It is called of class $\cC^{1,1}$ if $\bd K$ is $C^1$-smooth embedded hypersurface and the Gauss map $\bn_K:\bd K\to \S^d$ is Lipschitz continuous. We refer to \cite[Sec.\ 2.5]{Schneider:2014} for higher regularity and curvature notions on Euclidean convex bodies, which also translate directly to proper spherical convex bodies by applying the gnomonic projection.

For a proper spherical convex body $K\subset \S^d$, the classical \emph{isoperimetric inequality} (II) is
\begin{equation}\label{eqn:spherical_iso}
    P^s(K) \geq P^s(C_K),
\end{equation}
with equality if and only if $K$ is a geodesic ball, see, for example, \cite[Sec.~3.1.3]{AAGM:2015}. The isoperimetric inequality on the sphere can be traced back to Lévy \cite[p.\ 271]{Levy:1922} and E.~Schmidt \cite{Schmidt:1948} and a simple proof using spherical symmetrization was presented by Figiel, Lindenstrauss \& V.~Milman \cite[Thm.\ 2.1]{FLM:1977}.

Since the perimeter functional $P^s$ on convex bodies is increasing with respect to set inclusion, so in particular on geodesic balls, we may equivalently state \eqref{eqn:spherical_iso} in terms of the volume radii by
\begin{equation}\label{eqn:iso_rad}
    \alpha_P(K) \geq \alpha_K.
\end{equation}
where $\alpha_{P}(K)$ is uniquely determined radius such that 
\begin{equation*}
    P^s(K) = P^s(C(\alpha_{P}(K)))=\omega_{d-1} (\sin \alpha_P(K))^{d-1}.
\end{equation*}

Gao, Hug \& Schneider \cite{GHS:2003} gave an analog to the Blaschke--Santaló inequality for the dual volume. Unlike the Euclidean Blaschke--Santaló inequality, this spherical version does not depend on the choice of center. However, as they show, one can use the Euclidean Blaschke--Santaló inequality in the projective model with respect to the GHS-center to derive the spherical analog. Moreover, this spherical analog--which we call dual volume inequality (dVI)--can also equivalently be seen as a $U_1$-analog of the Urysohn inequality, where $U_1$ is another kind of volume functional that arises as expected hitting probability of $K$ with random $(d-1)$-dimensional greatspheres on $\S^d$, see, for example, \cite[Sec.\ 6.5]{SW:2008}. Note that the dual volume and $U_1$ are related by 
$\vol_d^s(K^*) = \omega_d(\frac{1}{2} - U_1(K))$.

\begin{theorem}[Dual Volume Inequality (dVI) / spherical Blaschke--Santaló Inequality \cite{GHS:2003}] \label{thm:BS}
Let $d\geq 2$ and let $K\subset \S^d$ be a proper spherical convex body. Then the following inequalities hold true and are all equivalent:
    \begin{enumerate}
        \item[(i)] $\vol_d(K^*)\leq \vol_d(C_K^*)$ 
        \item[(ii)] $\alpha_K+\alpha_{K^*}\leq \frac{\pi}{2}$
        \item[(iii)] $[\tan\alpha_K]\cdot[\tan\alpha_{K^*}]\leq 1$
    \end{enumerate}
    Equality holds if and only if $K$ is a geodesic ball.
\end{theorem}

In the following two propositions we combine recent results on Alexandrov--Fenchel type inequalities and the duality on the sphere. They are contained in the list of inequalities in \ Y.~Hu \& H.~Li \cite[Lem.\ 5.1]{HL:2023}, which we also recommend for an overview of the current state and references on the spherical Alexandrov--Fenchel type inequalities in spherical and hyperbolic space. By combing these inequalities and duality, we obtain two special inequalities that can be stated only using the perimeter and volume functionals. They are both stronger than the (dVI) in the sense that either combined with the classical isoperimetric inequality (II) implies the dual volume inequality (dVI), see the Remark \ref{rem:combine}. Also, there is seemingly no Euclidean analog for either of these two inequalities.

The first inequality, which we call dual isoperimetric inequality (dII), can be derived from the family of inequalities by C.~Chen, Guan, J.~Li \& Scheuer \cite[Prop.\ 4.4]{CGLS:2022} or by M.\ Chen \& Sun \cite[Prop.\ 3.3]{CS:2022}. They show for a proper spherical convex body of class $C_2^+$ that 
\begin{equation*}
    \cA_{k}(K) \geq \xi_{k,-1}(\cA_{-1}(K)) = \cA_k(C_K) \qquad \text{for all $0\leq k\leq d-2$},
\end{equation*}
with equality if and only if $K$ is a geodesic ball. Here $\cA_k$ are the family of spherical quermassintegrals inductively defined by
\begin{align*}
    \cA_{-1}(K) &= \vol_d^s(K),\\
    \cA_{0}(K) &= P^s(K),\\
    \cA_{1}(K) &= (d-1)\int_{\bd K} H_{1}^s(K,\bu) \, \vol_{\bd K}^s(\dint \bu) + (d-1) \vol_d^s(K),\\
    \cA_{k}(K) &= \tbinom{d-1}{k}\int_{\bd K} H_{k}^s(K,\bu) \, \vol_{\bd K}^s(\dint \bu) + \frac{d-k}{k-1} \cA_{k-2}(K), \qquad \text{for $k=2,\dotsc,d-1$.}
\end{align*}
Now for $k=d-3$ we have the duality relation
\begin{equation}\label{eqn:duality_rel_A}
    \alpha_{\cA_{d-3}}(K)+\alpha_P(K^*) = \frac{\pi}{2},
    \quad \text{or equivalently} \quad
    \cA_{d-3}(K) = (d-2) \left[\omega_{d-1}-P^s(K^*)\right].
\end{equation}
This yields our first inequality.

\begin{proposition}[Dual Isoperimetric Inequality (dII)]\label{cor:polar_iso}
    Let $d\geq 3$ and let $K\subset \S^d$ be a spherical convex body of class $\cC^2_+$. Then the following equivalent inequalities hold true:
    \begin{enumerate}
        
        \item[(i)] $P^s(K^*)\leq P^s(C_K^*)$
        
        \item[(ii)] $\alpha_{P}(K^*) + \alpha_K \leq \frac{\pi}{2}$
        
        \item[(iii)] $[\tan\alpha_K]\cdot[\tan \alpha_P(K^*)]\leq 1$
    \end{enumerate}
    Equality holds if and only if $K$ is a geodesic ball.
    
    For $d=2$ we have equality for all $K$, since $\alpha_K =
     \frac{\pi}{2}-\alpha_P(K^*)$.
\end{proposition}

\begin{remark}
     The identity \eqref{eqn:duality_rel_A} can be concluded from the duality relations for the $U_j$-functionals $U_j(K)+U_{d-j+1}(K^*)=\frac{1}{2}$, see \cite[Eq.\ (6.51) \& (6.63)]{SW:2008}. This was also obtained in \cite[Thm.\ 1.1]{HL:2023} with a different approach for smooth convex bodies in spherical and hyperbolic spaces. To verify that $\cA_k\sim W_{k+1} \sim U_{d-1-k}$ for $k=-1,\dotsc,d-1$, one can use \cite[Prop.\ 7]{Solanes:2006}. Indeed for the $\cA_k$ functionals used in \cite{CGLS:2022,CS:2022}, the $W_k$ functionals used in \cite{HL:2023} and the $U_j$ functionals used in \cite{SW:2008}, we have
    \begin{align*}
        \cA_{-1}(K) &= W_0(K) = \vol_d^s(K) = \omega_d U_d(K),\\
        \cA_{0}(K) &= d W_1(K) = P^s(K) = \omega_0\omega_{d-1} U_{d-1}(K),\\
        \cA_{1}(K) &= d(d-1) W_2(K) = \omega_1\omega_{d-2} U_{d-2}(K)
    \end{align*}
    and the recursive identities
    \begin{align*}
        \tbinom{d-1}{k} \int_{\bd K} H_{d-1-k}^s(K,\bu) \vol_{\bd K}^s(\dint\bu)
        &=\left[\cA_k(K)-\frac{d-k}{k-1} \cA_{k-2}(K)\right] \\
        &= d\tbinom{d-1}{k} \left[W_{k+1}(K)-\frac{k}{d+1-k} W_{k-1}(K)\right]\\
        &= \omega_{k} \omega_{d-1-k} \left[U_{d-1-k}(K) - U_{d+1-k}(K)\right],
    \end{align*}
    for $k=2,\dotsc,d-1$. One can verify inductively that
    \begin{equation}
        \cA_{k}(K) = d\tbinom{d-1}{k} W_{k+1}(K) = \omega_{k}\omega_{d-1-k}\,U_{d-1-k}(K).
    \end{equation}
    This and the duality relation now yields
    \begin{equation*}
        \cA_{d-3}(K) 
        = \omega_{2}\omega_{d-3} U_2(K) 
        = \omega_{2} \omega_{d-3}\left[\frac{1}{2} - U_{d-1}(K^*)\right]
        = (d-2) \left[\omega_{d-1}-P^s(K^*)\right].
    \end{equation*}
    Thus \eqref{eqn:duality_rel_A} is verified for $d\geq 3$.
\end{remark}

The second inequality follows from combining the main results of M.\ Chen \& Sun \cite[Thm.\ 1.2]{CS:2022} with duality. They show for a proper spherical convex body of class $C_2^+$ that 
\begin{equation*}
    \cA_{2\ell}(K) \geq \xi_{2\ell,0}(\cA_0(K)) = \cA_{2\ell}(C(\alpha_P(K))) \qquad \text{for all $1\leq \ell\leq \left\lfloor\frac{d-2}{2}\right\rfloor$},
\end{equation*}
with equality if and only if $K$ is a geodesic ball.
Now for $d>3$ odd, we find our second inequality by using the case $\ell=\frac{d-3}{2}$ and \eqref{eqn:duality_rel_A}.

\begin{proposition}[Dual Perimeter Inequality (dPI)]\label{cor:ppi}
    Let $d> 3$ odd and let $K\subset \S^d$ be a proper spherical convex body of class $\cC^2_+$. Then the following equivalent inequalities hold true:
    \begin{enumerate}
        
        \item[(i)] $P^s(K^*) \leq P^s(C(\alpha_P(K))^*)$
        
        \item[(ii)] $\alpha_P(K)+\alpha_P(K^*)\leq \frac{\pi}{2}$
        
        \item[(iii)] $[\tan\alpha_P(K)]\cdot[\tan\alpha_P(K^*)]\leq 1$
    \end{enumerate}
    Equality holds if and only if $K$ is a geodesic ball.
    
    For $d=3$ we have equality for all $K$, since $\alpha_P(K)=
    \frac{\pi}{2}-\alpha_P(K^*)$. For $d=2$ the inequality is reversed and equivalent to the isoperimetric inequality (II), since $\alpha_K = \frac{\pi}{2}-\alpha_P(K^*)$.
\end{proposition}

\begin{remark}
It is apparently still an open conjecture that Proposition \ref{cor:ppi} holds true also for all even $d\geq 4$, but the statement is equivalent to the Alexandrov--Fenchel type inequality $\cA_{d-3}(K)\geq \xi_{d-3,0}(\cA_0(K)) = \cA_{d-3}(C(\alpha_P(K)))$.
\end{remark}

\begin{remark}[dII and dPI on $\S^2$ and $\S^3$]
    For $d=2$ we have $\alpha_K + \alpha_P(K^*) = \frac{\pi}{2}$, or equivalently, $\vol_2^s(K) = 2\pi - P^s(K^*)$. Thus the dual isoperimetric inequality (dII) gives an equality for all spherical convex bodies on $\S^2$. Moreover, we have
    \begin{equation*}
        \alpha_K+\alpha_{K^*} = \pi - (\alpha_P(K)+\alpha_P(K^*)) \leq \frac{\pi}{2},
    \end{equation*}
    which yields $\alpha_P(K)+\alpha_P(K^*)\geq \frac{\pi}{2}$.
    
    For $d=3$ we have that $\alpha_P(K)+\alpha_P(K^*)=\frac{\pi}{2}$, or equivalently, $P^s(K^*)=4\pi-P^s(K)$ and therefore the dual isoperimetric inequality (dII) is equivalent to the isoperimetric inequality (II) on $\S^3$. 
\end{remark}

\begin{remark}[II + dII implies dVI and II + dPI implies dII] \label{rem:combine}
The isoperimetric inequality (II)
together with the dual isoperimetric inequality (dII) imply the dual volume inequality (dVI) by
\begin{equation*}
    \alpha_{K^*} \overset{\text{(II)}}\leq \alpha_P(K^*) \overset{\text{(dII)}}\leq \frac{\pi}{2}-\alpha_K,
    \quad \text{or equivalently,}\quad
    P^s(C_{K^*}) \overset{\text{(II)}}{\leq} P^s(K^*) \overset{\text{(dII)}}{\leq} P^s(C_K^*).
\end{equation*}

Moreover, for odd $d\geq 3$ we have
\begin{equation*}
    \alpha_P(K^*)\overset{\text{(dPI)}}\leq \frac{\pi}{2} - \alpha_P(K) \overset{\text{(II)}}\leq \frac{\pi}{2}-\alpha_{K},
    \quad \text{or equivalently,}\quad
    P^s(K^*) \overset{\text{(dPI)}}\leq P^s(C(\alpha_P(K)^*)) \overset{\text{(II)}}\leq P^s(C_K^*).
\end{equation*}
Hence the isoperimetric inequality (II) in conjunction with the dual perimeter inequality (dPI) implies the dual isoperimetric inequality (dII).
\end{remark}

\section{Isoperimetric inequalities for the spherical floating area}

For a convex body $K$ in $\mathbb{R}^d$ the classical Blaschke affine surface area \cite{Blaschke:1923, SW:1990} is defined as 
\begin{equation}\label{classical-asa}
    \as_1(K) =  \int_{\bd K} H_{d-1}^e(K,\bu)^{\frac{1}{d+1}} \, \vol_{\bd K}^e(\dint\bu),
\end{equation}
where $H^e_{d-1}(K,\cdot)$ is the generalized  Gauss--Kronecker curvature of $\bd K$ and $\vol_{\bd K}^e$ is the surface area measure of $\bd K$.
The spherical floating area $\Omega^s_1(K)$ of a spherical convex body $K$ was introduced in \cite{BW:2016,BW:2018} as a natural spherical analog to Blaschke's affine surface area. Like the classical Euclidean affine surface area, for which this was shown in  \cite{SW:1990}, the floating area arises in the volume derivative of the spherical floating body, see \cite[Thm.\ 2.1]{BW:2016} and can be expressed as a curvature measure, similar to  (\ref{classical-asa}) via
\begin{equation*}
    \Omega^s_1(K) = \int_{\bd K} H_{d-1}^s(K,\bu)^{\frac{1}{d+1}} \, \vol_{\bd K}^s(\dint\bu),
\end{equation*}
where $H^s_{d-1}(K,\cdot)$ is the generalized spherical Gauss--Kronecker curvature of $\bd K$ with respect to the outer unit normal vector $\bn_K$ on $\S^d$ and $\vol_{\bd K}^s$ is the surface area measure of $\bd K$, that is, the $(d-1)$-dimensional Hausdorff measure restricted to $\bd K$.
For a geodesic ball $C_{\bo}(\alpha) = \{\bv \in \S^d: d_s(\bo,\bv)\leq \alpha\}$ of radius $\alpha$ and center $\bo\in\S^d$ we have
\begin{equation}\label{eqn:floating_area_ball}
    \Omega^s_1(C_\bo(\alpha)) = \omega_{d-1} (\cos\alpha)^{\frac{d-1}{d+1}} (\sin\alpha)^{\frac{d(d-1)}{d+1}}.
\end{equation}

The affine isoperimetric inequality in $\mathbb{R}^d$ states that 
\begin{equation}\label{asa-inequality}
    \as_1(K) \leq d \vol_d^e(B_d^2)^{\frac{2}{d+1}} \vol_d^e(K)^{\frac{d-1}{d+1}} = \as_1(B_K),
\end{equation}
where $B^d_2$ is the centered Euclidean unit ball in $\mathbb{R}^d$ and $B_K$ is the center ball of the same volume as $K$. 
Equality holds if and only if $K$ is an ellipsoid.
It is natural to ask is the corresponding inequality holds for the spherical floating area $\Omega^s_1$. This is established in the next theorem.

\begin{theorem}[Spherical Floating Area Inequality]\label{thm:sfa-ineq}
    Let $d\geq 3$ and $K\subset \S^d$ be a spherical convex body and let $\bo=\bo(K)$ be the GHS-center.
    If $\overline{K}:=g_{\bo}(K)\subset \sqrt{d(d-2)}B_2^d$, then
    \begin{equation*}
        \Omega^s_1(K) \leq \Omega^s_1(C_K),
    \end{equation*}
    with equality if and only if $K$ is a geodesic ball.
\end{theorem}

If we fix a center $\be\in\S^d$ such that $K\subset  \interior H^+(\be)$, then we can express the spherical floating area of $K$ in terms of the projective model $\overline{K}:=g_{\be}(K)$ by
\begin{equation}\label{eqn:floating_area_projected}
    \Omega^s_1(K) = \Omega^s_1(\overline{K}) 
    = \int_{\bd \overline{K}} H_{d-1}^e(\overline{K},\bx)^{\frac{1}{d+1}} (1+\|\bx\|_2^2)^{-\frac{d-1}{2}}\, \vol_{\bd \overline{K}}^e(\dint\bx),
\end{equation}
see \cite[Eq.~(4.13)]{BW:2018}, where $H_{d-1}^e$ is the usual Euclidean Gauss--Kronecker curvature in $\R^d$ and $\vol_{\bd \overline{K}}^e$ denotes the Euclidean surface area measure of $\bd \overline{K}$. 

More generally, we denote by 
$\Sp^d(\lambda)$ the real space form of dimension $d$ and curvature $\lambda$. We can identify $\Sp^d(\lambda)$ with either a Euclidean sphere in $\R^{d+1}$ if $\lambda>0$, with $\R^d$ if $\lambda=0$, or with a hyperboloid in $\R^{d+1}_1$ if $\lambda<0$. Here we will only consider the cases $\lambda \geq 0$.
Then we have
 \begin{equation}\label{spaceforms}
    \Sp^d(\lambda) \cong \begin{cases}
                             \S^d(\lambda) := \{\bx\in \R^{d+1}: \bx\cdot \bx = 1/\lambda\}
                                 &\text{if $\lambda>0$},\\
                             \R^{d} &\text{if $\lambda=0$}.
                         \end{cases}
 \end{equation}
 A convex body $K\subset \Sp^d(\lambda)$ is a compact geodesically convex subset.

 Thus the $\lambda$-floating area $\Omega^\lambda(K)$ of a proper convex body $K\subset \Sp^d(\lambda)$ in a space form of constant curvature $\lambda$ can be expressed via the projective model $\overline{K}\subset \R^d$ by
\begin{equation}\label{eqn:lambda_floating_area_projected}
    \Omega^\lambda(K) = \Omega^\lambda(\overline{K}) = \int_{\bd \overline{K}} H_{d-1}^e(\overline{K},\bx)^{\frac{1}{d+1}} (1+\lambda\|\bx\|_2^2)^{-\frac{d-1}{2}}\, \vol_{\bd \overline{K}}^e(\dint\bx),
\end{equation}
see \cite[Eq.~(4.13)]{BW:2018}. In particular, let $C_\bo(\alpha)= \{\bv \in \Sp^d(\lambda): d_\lambda(\bo,\bv)\leq \alpha\}$ be a geodesic ball in $\Sp^d(\lambda)$.
Here
$d_{\lambda}(\bo,\bv)$ denotes  the minimal geodesic distance in $\Sp^d(\lambda)$ of $\bo$ to $\bv$. Then
\begin{equation}\label{eqn:FA_ball_lambda}
    \Omega^\lambda(C_\bo(\alpha)) = \omega_{d-1} (\cos_\lambda \alpha)^{\frac{d-1}{d+1}} (\sin_\lambda\alpha)^{\frac{d(d-1)}{d+1}},
\end{equation}
where for $\lambda >0$ we set
\begin{align}\label{eqn:lambda_sin}
    \cos_\lambda t &:= \cos\sqrt{\lambda} t, & 
    \sin_\lambda t &:= \frac{\sin \sqrt{\lambda} t}{\sqrt{\lambda}}, &
    \tan_\lambda t &:= \frac{\sin_\lambda t}{\cos_\lambda t} = \frac{1}{\sqrt{\lambda}} \, \frac{\sin \sqrt{\lambda} t}{\cos\sqrt{\lambda} t}.
\end{align}
Note that $(\cos_\lambda t)^2+\lambda (\sin_\lambda t)^2 = 1$ and
\begin{align*}
    (\cos_\lambda t)' &= -\lambda \sin_\lambda t, &
    (\sin_\lambda t)' &= \cos_\lambda t, &
    (\tan_\lambda t)' &= 1+\lambda(\tan_\lambda t)^2 = \frac{1}{(\cos_\lambda t)^2}.
\end{align*}

We prove Theorem \ref{thm:sfa-ineq} in a more general version by considering a real space form $\Sp^d(\lambda)$ of constant curvature $\lambda>0$.

\begin{theorem}[$\lambda$-Floating Area Inequality] \label{thm:lambda_FI}
    Let $d\geq 3$, $\lambda>0$ and $K\subset \Sp^d(\lambda)$ be a convex body and let $\bo=\bo(K)$ be the GHS-center.
    If for the projective model $\overline{K}:=g_{\bo}^\lambda(K)$ we have $\overline{K} \subset \sqrt{\frac{d(d-2)}{\lambda}}B_2^d$, then 
    \begin{equation*}
        \Omega^\lambda_1(K) \leq \Omega^\lambda_1(C_K^\lambda),
    \end{equation*}
    with equality if and only if $K$ is a geodesic ball.
\end{theorem}
\begin{proof}
    We have, by \eqref{eqn:lambda_floating_area_projected} and Hölder's inequality,
    \begin{align}\notag
        \Omega^\lambda_1(K) 
            &= \int_{\bd \overline{K}} \left(\frac{H_{d-1}^e(\overline{K},\bx)}{(\bx\cdot \bn_{\overline{K}}(\bx))^d}\right)^{\frac{1}{d+1}} \frac{(\bx\cdot \bn_{\overline{K}}(\bx))^{\frac{d}{d+1}}}{(1+\lambda\|\bx\|^2)^{\frac{d-1}{2}}}\, \vol_{\bd \overline{K}}^{e}(\dint \bx)\\ \notag
        &\leq \left(\int_{\bd \overline{K}} \frac{H_{d-1}^e(\overline{K},\bx)}{(\bx\cdot \bn_{\overline{K}}(\bx))^d} 
        \vol_{\bd \overline{K}}^e(\dint \bx)\right)^{\frac{1}{d+1}} 
            \left(\int_{\bd \overline{K}} \frac{\bx\cdot \bn_{\overline{K}}(\bx)}{(1+\lambda\|\bx\|^2)^{\frac{d^2-1}{2d}}} 
                \,\vol_{\bd \overline{K}}^e(\dint\bx)\right)^{\frac{d}{d+1}}\\\label{eqn:proof_iso_ineq}
        &\leq \left(d\vol_d^e(\overline{K}^\circ)\right)^{\frac{1}{d+1}} \left(\int_{\S^{d-1}} \frac{\rho_{\overline{K}}(\bu)^d}{(1+\lambda\rho_{\overline{K}}(\bu)^2)^{\frac{d^2-1}{2d}}}\,\dint\bu\right)^{\frac{d}{d+1}}.
    \end{align}
     In the final inequality we used the transformation $\bx=\rho_{\overline{K}}(\bu)\bu$ for the second factor, see for example \cite[Lem.\ 3.1]{Hug2:1996}. For the first factor we use the transformation $\bu=\bn_{\overline{K}}(\bx)$.  We can restrict the domain of integration to normal boundary points $\bx\in\bd \overline{K}$ with $H^e_{d-1}(\overline{K},\bx)>0$, see \cite[Sec.~2]{Hug1:1996}, to obtain the inequality
     \begin{equation*}
         \int_{\bd \overline{K}} \frac{H_{d-1}^e(\overline{K},\bx)}{(\bx\cdot \bn_{\overline{K}}(\bx))^d} \vol_{\bd \overline{K}}^e(\dint \bx)
         \leq \int_{\S^{d-1}} [h_{\overline{K}}(\bu)]^{-d}\,\dint\bu 
         = \int_{\S^{d-1}} [\rho_{\overline{K}^\circ}(\bu)]^{d}\, \dint\bu = d\vol_{d}^e(\overline{K}^\circ).
     \end{equation*}

    For the $\lambda$-volume, see \cite[Eq.~(3.20)]{BW:2018}, we have
    \begin{equation*}
        \vol_d^\lambda(K) = \int_{\overline{K}} (1+\lambda\|\bx\|^2)^{-\frac{d+1}{2}}\, \dint\bx
        = \int_{\S^{d-1}} \int_{0}^{\rho_{\overline{K}}(\bu)} \frac{r^{d-1}}{(1+\lambda r^2)^{\frac{d+1}{2}}} \,\dint r\, \dint\bu
        = \int_{\S^{d-1}} J_\lambda(\arctan_{\lambda}(\rho_{\overline{K}}(\bu)))\, \dint\bu,
    \end{equation*}
    where $J_\lambda:[0,\pi/(2\sqrt{\lambda})] \to \R$ is defined by 
    \begin{equation*}
        J_\lambda(\alpha) := \int_{0}^\alpha (\sin_\lambda t)^{d-1}\, \dint t.
    \end{equation*}

    Following ideas of Gao, Hug and Schneider \cite{GHS:2003}, 
    we consider
    \begin{equation*}
        H_\lambda(t):=(\tan_\lambda J_\lambda^{-1}(t))^d.
    \end{equation*}
    Then
    \begin{equation}
        H_\lambda'(t) = d (\tan_\lambda J_\lambda^{-1}(t))^{d-1} \frac{(1+\lambda(\tan_\lambda J_\lambda^{-1}(t))^2)}{(\sin_\lambda J_\lambda^{-1}(t))^{d-1}} = \frac{d}{(\cos_\lambda J_\lambda^{-1}(t))^{d+1}} \geq d
    \end{equation}
    and
    \begin{equation}
        H_\lambda''(t) = \frac{\lambda d(d+1)}{(\cos_\lambda J_\lambda^{-1}(t))^{d+2} (\sin_\lambda J_\lambda^{-1}(t))^{d-2}} \geq \lambda d(d+1).
    \end{equation}
    So $H_\lambda$ is strictly increasing and strictly convex and we conclude by Jensen's inequality that
    \begin{align}
        H_\lambda\left(\frac{\vol_d^\lambda(K)}{\omega_{d-1}}\right) 
        &\leq \frac{1}{d\kappa_d} \int_{\S^{d-1}} H_\lambda(J_\lambda(\arctan_\lambda(\rho_{\overline{K}}(u)))) \, \dint\bu
        = \frac{1}{d\kappa_d} \int_{\S^{d-1}} \rho_{\overline{K}}(u)^d \, \dint\bu = \frac{\vol_d^e(\overline{K})}{\kappa_d}, 
    \end{align}
    with equality if and only if $K$ is a spherical cap centered in $\bo$, or equivalently if $\overline{K}$ is a Euclidean ball centered in the origin.
    Thus, for the geodesic ball $C_K=C(\alpha_K^\lambda,\bo)$ such that $\vol_d^\lambda(K) = \vol_d^\lambda(C_K^\lambda)$, we find $\overline{C}_K^\lambda:=g_\bo^\lambda(C_K^\lambda) = (\tan_\lambda \alpha_K^\lambda)B_2^d$ and
    \begin{equation*}
        (\tan_\lambda\alpha_K^\lambda)^d = \frac{\vol_d^e(\overline{C}_K^\lambda)}{\kappa_d} = H_\lambda\left(\frac{\vol_d^\lambda(C_K^\lambda)}{\omega_{d-1}}\right) 
        = H_\lambda\left(\frac{\vol_d^\lambda(K)}{\omega_{d-1}}\right) \leq \frac{\vol_d^e(\overline{K})}{\kappa_d}.
    \end{equation*}
    Since $\bo$ is the GHS-center, $\overline{K}=g_{\bo}^\lambda(K)$ is centered.
    Thus by the classical Blaschke--Santaló inequality we conclude first
    \begin{equation*}
        d\vol_d^e(\overline{K}^\circ) \leq \omega_{d-1} \frac{\kappa_d}{\vol_d^e(\overline{K})} \leq \omega_{d-1} (\tan_\lambda\alpha_K^\lambda)^{-d}.
    \end{equation*}
    For the second factor we write
    \begin{equation*}
        \int_{\S^{d-1}} \frac{\rho_{\overline{K}}(\bu)^d}{(1+\lambda\rho_{\overline{K}}(\bu)^2)^{\frac{d^2-1}{2d}}}\, \dint\bu 
        = \int_{\S^{d-1}} G_\lambda(J_\lambda(\arctan_\lambda(\rho_{\overline{K}}(\bu)))) \, \dint\bu,
    \end{equation*}
    where
    \begin{equation*}
        G_\lambda(t) 
        := \frac{[\tan_\lambda J_\lambda^{-1}(t)]^d}{(1+\lambda[\tan_\lambda J_\lambda^{-1}(t)]^2)^{\frac{d^2-1}{2d}}} 
        = [\sin_\lambda J_\lambda^{-1}(t)]^d [\cos_\lambda J_\lambda^{-1}(t)]^{-1/d}.
    \end{equation*}
    We calculate
    \begin{align*}
        G_\lambda'(t) 
        &= d [\cos_\lambda J_\lambda^{-1}(t)]^{(d-1)/d} + (\lambda/d) [\sin_\lambda J_\lambda^{-1}(t)]^2 [\cos_\lambda J_\lambda^{-1}(t)]^{-(d+1)/d} \\
        &= d [\cos_\lambda J_\lambda^{-1}(t)]^{-(d+1)/d} \left(1 - \lambda \frac{d^2-1}{d^2} [\sin J^{-1}(t)]^2\right)
    \end{align*}
    and
    \begin{align*}
        G_\lambda''(t) 
        &= \lambda (d+1) [\cos_\lambda J_\lambda^{-1}(t)]^{-(2d+1)/d} [\sin_\lambda J_\lambda^{-1}(t)]^{-(d-2)} \left(1 - \lambda \frac{d^2-1}{d^2} [\sin_\lambda J_\lambda^{-1}(t)]^2\right)\\
            &\qquad - 2\lambda\frac{d^2-1}{d} [\cos_\lambda J_\lambda^{-1}(t)]^{-1/d} [\sin_\lambda J_\lambda^{-1}(t)]^{-(d-2)}\\
        &= \lambda(d+1) [\cos_\lambda J_\lambda^{-1}(t)]^{-(2d+1)/d} [\sin_\lambda J_\lambda^{-1}(t)]^{-(d-2)} \\
            &\qquad \left[ 1- \lambda\frac{d^2-1}{d^2} [\sin_\lambda J_\lambda^{-1}(t)]^2 - 2\frac{d-1}{d} [\cos_\lambda J_\lambda^{-1}(t)]^2\right]\\
        &= -\lambda\frac{(d+1)(d-2)}{d} [\cos_\lambda J_\lambda^{-1}(t)]^{-(2d+1)/d} [\sin_\lambda J_\lambda^{-1}(t)]^{-(d-2)} \left[ 1 -\lambda\frac{(d-1)^2}{d(d-2)}[\sin_\lambda J_\lambda^{-1}(t)]^2 \right].
    \end{align*}
    Thus, for $d\geq 3$, $G_\lambda''(t) \leq 0$ if and only if 
    \begin{align*}
        \sin_\lambda J_\lambda^{-1}(t) \leq \frac{1}{\sqrt{\lambda}}\,\frac{\sqrt{d(d-2)}}{d-1} \quad \Leftrightarrow  \quad 
        \tan_\lambda J_\lambda^{-1}(t) = \frac{\sin_\lambda J_\lambda^{-1}(t)}{\sqrt{1-\lambda[\sin_\lambda J_\lambda^{-1}(t)]^2}} \leq \sqrt{\frac{d(d-2)}{\lambda}}.
    \end{align*}
  Since $\rho_{\overline{K}}(\bu) \leq \sqrt{d(d-2)/\lambda}$ for all $\bu\in\S^{n-1}$,  we have that $G_\lambda$ is concave and by Jensen's inequality we get that
    \begin{align}\notag
        \frac{1}{\omega_{d-1}} \int_{\S^{d-1}} \frac{\rho_{\overline{K}}(\bu)^d}{(1+\lambda\rho_{\overline{K}}(\bu)^2)^{\frac{d^2-1}{2d}}}\, \dint\bu
        &\leq G_\lambda\left( \frac{1}{\omega_{d-1}} \int_{\S^{d-1}} J_\lambda(\arctan_\lambda(\rho_{\overline{K}}(\bu))) \, \dint \bu\right)\\
        &= G_\lambda\left(\frac{\vol_d^\lambda(K)}{\omega_{d-1}}\right) \notag
        = G_\lambda\left(\frac{\vol_d^\lambda(C_K^\lambda)}{\omega_{d-1}}\right) =G_\lambda(J_\lambda(\alpha_K^\lambda))\\
        &= \frac{(\tan_\lambda \alpha_K^\lambda)^d}{(1+\lambda(\tan_\lambda \alpha_K^\lambda)^2)^{\frac{d^2-1}{2d}}} 
        = (\sin_\lambda \alpha_K^\lambda)^d (\cos_\lambda \alpha_K^\lambda)^{-1/d}.\label{eqn:equality_case}
    \end{align}
 
 Combing all of the above and \eqref{eqn:FA_ball_lambda}, we conclude
    \begin{align*}
        \Omega^\lambda_1(K) 
        &\leq \omega_{d-1} (\tan_\lambda\alpha_K^\lambda)^{-\frac{d}{d+1}} (\sin_\lambda \alpha_K^\lambda)^{\frac{d^2}{d+1}} 
            (\cos_\lambda \alpha_K^\lambda)^{-\frac{1}{d+1}}\\
        &= \omega_{d-1} (\cos_\lambda \alpha_K^\lambda)^{\frac{d-1}{d+1}} (\sin_\lambda \alpha_K^\lambda)^{\frac{d(d-1)}{d+1}} = \Omega^\lambda_1(C_K^\lambda).
    \end{align*}
    By \eqref{eqn:equality_case}, we have equality if and only if $\rho_{\overline{K}}(\bu) = \tan_\lambda \alpha_K^\lambda$ for all $\bu\in\S^{d-1}$ since $\rho_{\overline{K}}$ is continuous. Thus equality holds if and only if $K=C_K^\lambda$ is a geodesic ball.
\end{proof}

\begin{remark}[the Euclidean case for $\lambda\to 0^+$]
    For $\lambda\to 0^+$ we have
    \begin{equation*}
        \alpha_K^\lambda = J_\lambda^{-1}\left(\frac{\vol_d^\lambda(K)}{\omega_{d-1}}\right) \to \left(\frac{\vol_d^e(K)}{\kappa_d}\right)^{\frac{1}{d}},
    \end{equation*}
    which also yields
    \begin{equation*}
        \Omega^\lambda_1(C_K^\lambda) = \omega_{d-1} (\cos_\lambda\alpha_K)^{\frac{d-1}{d+1}} (\sin_\lambda\alpha_K)^{\frac{d(d-1)}{d+1}} \to d\, \kappa_d^{\frac{2}{d+1}} \vol_d^e(K)^{\frac{d-1}{d+1}} = \as_1(B_K).
    \end{equation*}
    For $d\geq 3$ Theorem \ref{thm:lambda_FI} shows that the classical affine isoperimetric inequality belongs to a real-analytic family of isoperimetric inequalities. Remarkably, for $d=2$ the above proof of Theorem \ref{thm:lambda_FI} does not appear to work.
\end{remark}

\begin{remark}[Progress on the conjectured $\lambda$-floating area inequality for $\lambda>0$]
    For $d\geq 2$ and $\lambda>0$ let $K\subset\Sp^d(\lambda)$ be a proper convex body. Since $\Omega^\lambda$ is upper semi-continuous and by compactness there is a convex body $L_0$ that maximizes
    \begin{equation*}
        \Omega^\lambda_1(L_0) = \max \left\{ \Omega^\lambda(L) : \vol_d^\lambda(L)=\vol_d^\lambda(K)\right\}.
    \end{equation*}
    We conjecture that $L_0=C_K^\lambda$ is a geodesic ball of the same $\lambda$-volume as $K$. Theorem \ref{thm:lambda_FI}  verifies this if $K$ can be bounded with respect to the GHS-center $\bo$ by a geodesic ball of radius
    \begin{equation*}
        R(\lambda) := \frac{1}{\sqrt{\lambda}} \arctan \sqrt{\frac{d(d-2)}{\lambda}} \geq \frac{\pi}{2\sqrt{\lambda}} - \frac{1}{\sqrt{d(d-2)}}.
    \end{equation*}
    Since a proper convex body $K$ is by definition contained in the interior of a half-space, i.e., a geodesic ball of radius $\frac{\pi}{2\sqrt{\lambda}}$, we see that asymptotically for $d\to\infty$ we have $R(\lambda)\to \frac{\pi}{2\sqrt{\lambda}}$. So heuristically the obstruction imposed by $R(\lambda)$ vanishes asymptotically as the dimension increases.
 
    Also note that for $\lambda=1$ and $d\geq 3$, in Corollary \ref{cor:p-float_ineq} we find that 
    \begin{equation*}
        \Omega^s_1(K)\leq \Omega^s_1(C_K),
    \end{equation*}
    if $K\subset \S^d$ is a proper spherical convex body of class $\cC^2_+$ and such that $K\supset C_{\bo}(\arctan\sqrt{d})$.

    Finally, we mention that  a proof of the isoperimetric inequality for the spherical case $\lambda=1$ and $d=2$ was very recently established by Basit, Hoehner, Lángi \& Ledford \cite[Thm.\ 3.8]{BHLL:2024}: if $K\subset \S^2$ is a proper spherical convex body of class $\cC^2_+$ that is symmetric about a fixed center $\bo\in\S^d$. Their approach is completely different from our proof of Theorem \ref{thm:lambda_FI} and instead looks at optimal polygonal approximation of spherical convex bodies in terms of the volume difference.
\end{remark}

\section{Stability for the dual volume and floating area on the sphere} \label{sec:stab}

For a proper spherical convex body $K\subset \S^d$ we have seen that the gnomonic projection with respect to the GHS-center $\bo=\bo(K)$ gives a projective model $\overline{K}:=g_{\bo}(K)\subset \R^d$ of $K$ that is centered. 
Since the function $H(t) = (\tan J^{-1}(t))^d$ is strictly convex, we find for the Euclidean volume of $\overline{K}$ that
\begin{equation*}
    \frac{\vol_d^e(\overline{K})}{\kappa_d}
    = \frac{1}{\omega_{d-1}} \int_{\S^{d-1}} \rho_{\overline{K}}(\bu)\,\dint\bu 
    = \frac{1}{\omega_{d-1}} \int_{\S^{d-1}} H(J(\arctan \rho_{\overline{K}}(\bu)))\, \dint \bu
    \geq H\left(\frac{\vol_d^s(K)}{\omega_{d-1}}\right).
\end{equation*}
This inequality was used by Gao, Hug \& Schneider \cite{GHS:2003} to proof the spherical Blaschke--Santaló inequality. Recently a stability version of this inequality was established by Hug \& Reichenbacher \cite{HR:2017}, which we reformulate in the following

\begin{lemma}[{Stability of the projected volume \cite[Eq.~(5)]{HR:2017}}]\label{lem:HR}
Let $d\geq 2$, $K\subset \S^d$ be a proper spherical convex body and let $\bo\in(\interior K)\cap (\interior K^*)$. Then
    \begin{equation*}
        \frac{\vol_d^e(g_{\bo}(K))}{\kappa_d} \geq (1+\beta\Delta_2(K)^2)\, H\!\left(\frac{\vol_d^s(K)}{\omega_{d-1}}\right),
    \end{equation*}
    where
    \begin{equation}\label{def:beta}
        \beta := \frac{d(d+1)}{2 H(\vol_d^s(K)/\omega_{d-1})} = \frac{d(d+1)}{2(\tan\alpha_K)^d}.
    \end{equation}
    and
    \begin{equation}\label{def:Delta2}
        \Delta_2(K)^2 := \frac{1}{\omega_{d-1}}\int_{\S^{d-1}} [J(\rho_{\bo}(K,\bu))-J(\alpha_K)]^2\, \dint\bu
    \end{equation}
\end{lemma}

We derive the following
\begin{corollary}
    Let $d\geq 2$, let $K\subset \S^d$ be a proper spherical convex body and let $\bo\in(\interior K)\cap (\interior K^*)$. If for some $\varepsilon>0$ 
    \begin{equation*}
        \frac{\vol_d^e(g_{\bo}(K))}{\kappa_d} \leq (1+\varepsilon) \, H\!\left(\frac{\vol_d^s(K)}{\omega_{d-1}}\right),
    \end{equation*}
    then
    \begin{equation*}
        \vol_d^s(K\triangle C_K(\bo)) \leq \omega_{d-1}\sqrt{\frac{\varepsilon}{\beta}},
    \end{equation*}
    where $\triangle$ denotes the symmetric difference, i.e., $A\triangle B = (A\setminus B) \cup (B\setminus A)$.
\end{corollary}
\begin{proof}
    This follows from Lemma \ref{lem:HR} since
    \begin{align} \notag
        \Delta_2(K) 
        &\geq \frac{1}{\omega_{d-1}} \int_{\S^{d-1}} \left|J(\rho_{\bo}(K,\bu))-J(\alpha_K)\right|\,\dint\bu\\ \notag
        &= \frac{1}{\omega_{d-1}} \int_{\S^{d-1}} J(\rho_{\bo}(K\cup C_K(\bo),\bu) - J(\rho_{\bo}(K\cap C_K(\bo),\bu) \,\dint\bu\\ \label{eqn:Delta_2_lower}
        &= \frac{\vol_d^s(K\cup C_K(\bo))-\vol_d^s(K\cap C_K(\bo))}{\omega_{d-1}}
        = \frac{\vol_d^s(K\triangle C_K(\bo))}{\omega_{d-1}}.
    \end{align}
    Thus, by Lemma \ref{lem:HR},
    \begin{equation*}
        \left(1+\beta \frac{\vol_d^s(K\triangle C_K(\bo))^2}{\omega_{d-1}^2}\right) H\left(\frac{\vol_d^s(K)}{\omega_{d-1}}\right) \leq \frac{\vol_d^e(g_{\bo}(K))}{\kappa_d} \leq (1+\varepsilon) H\left(\frac{\vol_d^s(K)}{\omega_{d-1}}\right).
    \end{equation*}
    This concludes the proof.
\end{proof}

Hug \& Reichenbacher \cite{HR:2017} use Lemma \ref{lem:HR} in the GHS-center to establish the following stability version of the spherical Blaschke--Santaló inequality due to Gao, Hug \& Schneider \cite{GHS:2003}. Below we give a slightly modified version of \cite[Thm.\ 4.1]{HR:2017}.

\begin{theorem}[Stability for the dual volume \cite{HR:2017}] \label{thm:stab_HR}
Let $d\geq 2$ and let $K\subset \S^d$ be a proper spherical convex body. If for some $\varepsilon\in(0,1)$
\begin{equation*}
     \vol_d^s(K^*) \geq (1-\varepsilon) \vol_d^s(C_K^*),
\end{equation*}
then
\begin{equation*}
    \vol_d^s(K\triangle C_K(\bo)) \leq  \omega_{d-1} \gamma \sqrt{\varepsilon},
\end{equation*}
where $\bo=\bo(K)$ is the GHS-center and
$\gamma :=  \frac{\pi}{2} \sqrt{ \frac{1}{\sin\alpha_K} \left(1 + \frac{8}{d(d+1)\pi^2} (\tan\alpha_K)^{d}\right) } $.
\end{theorem}

\begin{proof}
    Since $\bo=\bo(K)$ is the GHS-center, we have that $\overline{K}=g_{\bo}(K)$ has centroid in the origin.
    Hug \& Reichenbacher \cite[p.\ 10]{HR:2017} show that
    \begin{equation}\label{eqn:HR-stab1}
        \vol_d^s(K^*) \leq (1-\beta_3\Delta_2(K)^2) \vol_d^s(C_K(\bo)^*),
    \end{equation}
    where 
    \begin{align*}
        \beta_1 &:= \frac{d(d+1)}{2(\tan\alpha_K)^{d}}, &
        \beta_2 &:= \frac{1}{(\pi/2)^2 + 1/\beta_1}, &
        \beta_3 &:= \min \left\{\frac{\beta_2}{d} \frac{(\sin\alpha_K)^{d+1}}{(\tan\alpha_K)^{d} J(\frac{\pi}{2}-\alpha_K)}, \frac{4}{\pi^2}\right\},
    \end{align*}
    and $\Delta_2(K)$ is defined as before, see \eqref{def:Delta2}.

    We  verify that \eqref{eqn:HR-stab1} implies the statement of the theorem. First, we show that $F(s):=\frac{J(s)}{(\sin s)^{d}}$ is increasing. We claim
    \begin{equation}\label{eqn:proof1}
        F'(s) = \frac{1-d(\cos s)F(s)}{\sin s} \overset{?}{\geq} 0, \qquad \text{for all $0\leq s \leq \frac{\pi}{2}$}.
    \end{equation}
    To see this consider $G(s) := \frac{1}{d}F'(s)(\tan s) (\sin s)^d = \frac{1}{d}(\tan s)(\sin s)^{d-1}-J(s)$ and observe that
    \begin{equation*}
        G'(s) = \frac{1}{d} (\tan s)^2(\sin s)^{d-1} \geq 0.
    \end{equation*}
    Thus $G(s)\geq G(0)=0$ for all $s\in[0,\frac{\pi}{2}]$, which implies \eqref{eqn:proof1} as claimed. Moreover, \eqref{eqn:proof1} yields
    \begin{equation*}
        F(s) \leq \frac{1}{d\cos s},
    \end{equation*}
    and therefore
    \begin{equation*}
        \frac{J(\frac{\pi}{2}-\alpha_K)}{(\cos\alpha_K)^d} = F\left(\frac{\pi}{2}-\alpha_K\right) \leq \frac{1}{d \sin \alpha_K}.
    \end{equation*}
    Hence
    \begin{equation*}
        \beta_3 \geq \min\left\{\beta_2 (\sin\alpha_K)^2,\frac{4}{\pi^2}\right\} = \frac{4}{\pi^2} \min\left\{ \frac{(\sin\alpha_K)^2}{1+ \frac{8}{d(d+1)\pi^2} (\tan\alpha_K)^d}, 1\right\}.
    \end{equation*}
    Now
    \begin{equation*}
        \frac{(\sin\alpha_K)^2}{1 + \frac{8}{d(d+1)\pi^2} (\tan\alpha_K)^d} \leq \frac{1}{1 + \frac{8}{d(d+1)\pi^2} (\tan\alpha_K)^d} \leq 1.
    \end{equation*}
    We conclude that
    \begin{equation*}
        \beta_3 \geq \frac{4}{\pi^2} \frac{(\sin\alpha_K)^2}{1 + \frac{8}{d(d+1)\pi^2} (\tan\alpha_K)^d} =\frac{1}{\gamma^2}.
    \end{equation*}
    By the above and \eqref{eqn:HR-stab1}, we obtain that
    \begin{align*}
        [1-(\Delta_2(K)/\gamma)^2] \vol_d^s(C_K(\bo)^*) 
        &\geq \vol_d^s(K^*) 
        \geq (1-\varepsilon) \vol_d^s(C_K(\bo)^*),
    \end{align*}
    which yields
    \begin{equation*}
        \gamma \sqrt{\varepsilon} \geq \Delta_2(K) \geq \frac{\vol_d^s(K\triangle C_K(\bo))}{\omega_{d-1}},
    \end{equation*}
    where the final inequality follows by \eqref{eqn:Delta_2_lower}.
\end{proof}
Using Lemma \ref{lem:HR} we are able to derive the following stability version for the isoperimetric inequality for $\Omega^s_1$.

\begin{theorem}[Stability for the floating area]\label{thm:stability_fa}
    Let $d\geq 3$ and let $K\subset \S^d$ be a proper spherical convex body such that $\overline{K}:=g_{\bo}(K)\subset\sqrt{d(d-2)}B_2^d$, where $\bo=\bo(K)$ is the GHS-center. If for some $\varepsilon\in(0,\frac{1}{d+1})$ 
    \begin{equation*}
        \Omega^s_1(K) \geq (1-\varepsilon) \Omega^s_1(C_K),
    \end{equation*}
    then
    \begin{equation*}
        \vol_d^s(K\triangle C_K(\bo)) \leq \omega_{d-1} \tau \sqrt{\varepsilon},
    \end{equation*}
    where $\tau := \sqrt{3H(\vol_d^s(K)/\omega_{d-1})} =\sqrt{3(\tan\alpha_K)^d} \leq 2d^{d/2}$.
\end{theorem}
\begin{proof}
    For $\lambda=1$ we use the same definitions as in the proof of Theorem \ref{thm:lambda_FI}.
    By \eqref{eqn:proof_iso_ineq} and \eqref{eqn:equality_case} we conclude
    \begin{align*}
        \frac{\Omega^s_1(K)}{\omega_{d-1}} \leq \left(\frac{\vol_d^s(\overline{K}^\circ)}{\kappa_d}\right)^{\frac{1}{d+1}} G\left(\frac{\vol_d^s(K)}{\omega_{d-1}}\right)^{\frac{d}{d+1}}
    \end{align*}
    Since $\bo$ is the GHS-center, we have that $\overline{K}$ is centered and by the Blaschke--Santaló inequality and Lemma \ref{lem:HR} we derive
    \begin{equation*}
        \frac{\vol_d^e(\overline{K}^\circ)}{\kappa_d} \leq \frac{\kappa_d}{\vol_d^e(\overline{K})} \leq \frac{1}{(1+\beta\Delta_2(K)^2) H(\vol_d^s(K)/\omega_{d-1})}
        =\frac{(\tan\alpha_K)^{-d}}{1+\beta\Delta_2(K)^2}.
    \end{equation*}
    Note that
    \begin{align*}
        \frac{\Omega^s_1(C_K)}{\omega_{d-1}}
        &= \left(\frac{\vol_d^e(\overline{C}_K^\circ)}{\kappa_d}\right)^\frac{1}{d+1} G\left(\frac{\vol_d^s(C_K)}{\omega_{d-1}}\right)^{\frac{d}{d+1}}
        = (\tan\alpha_K)^{-\frac{d}{d+1}}\, G\left(\frac{\vol_d^s(K)}{\omega_{d-1}}\right)^{\frac{d}{d+1}}.
    \end{align*}
    By $(1-\varepsilon)\Omega^s_1(C_K)\leq \Omega^s_1(K)$ and the above, we conclude
    \begin{equation*}
        (1-\varepsilon)^{d+1} \leq \frac{1}{1+\beta\Delta_2(K)^2}.
    \end{equation*}
    This,  together with \eqref{eqn:Delta_2_lower}, yields
    \begin{equation*}
        \frac{\vol_d^s(K\triangle C_K(\bo))}{\omega_{d-1}} 
        \leq \Delta_2(K) 
        \leq \sqrt{\frac{1}{\beta} \left[\frac{1}{(1-\varepsilon)^{d+1}}-1\right]}
        \leq \sqrt{\frac{3(d+3)\varepsilon}{\beta}},
    \end{equation*}
    where in the last inequality we used that for $\varepsilon\leq \frac{1}{d+1}$ and $d\geq 3$
    \begin{equation*}
        Z(\varepsilon):=\frac{1}{(1-\varepsilon)^{d+1}}-1-3(d+3)\varepsilon\leq 0. 
    \end{equation*}
    This holds  as  $Z(0)=0$ and
    \begin{equation*}
        Z'(\varepsilon) = \frac{d+1}{(1-\varepsilon)^{d+2}} - 3(d+3) \leq d (1+1/d)^{d+3} - 3(d+3)<0,
    \end{equation*}
    for all $d\geq 3$ and $\varepsilon\leq \frac{1}{d+1}$.
    The theorem now follows since $\sqrt{3(d+3)/\beta}\leq \tau$ and $H(\vol_d^s(K)/\omega_{d-1}) =(\tan\alpha_K)^{d}\leq (d(d-2))^{d/2}$.
\end{proof}

\section{The \texorpdfstring{$L_p$}{Lp}-floating area on the sphere}

In the Euclidean setting, the $L_p$-affine surface area was introduced as  natural extension of the affine surface by Lutwak \cite{Lutwak:1996}, and in \cite{Hug1:1996, Ludwig:2010, SW:2004, WY:2008}.
We introduce  a corresponding quantity in the spherical setting.  
See also Section \ref{sec:real_limit}.
For $p\in\R\setminus\{-d\}$ and a proper spherical convex body $K\subset \S^d$, where we further assume that $K$ is of class $\cC^2_+$ if $p<0$, we define the \emph{$L_p$-floating area} by
\begin{equation} \label{lpfloatarea}
    \Omega_p^s(K) = \int_{\bd K} H_{d-1}^s(K,\bu)^{\frac{p}{d+p}}\, \vol_{\bd K}^s(\dint \bu).
\end{equation}
We note that $\Omega_0^s(K) = P^s(K) = \vol_{\bd K}^s(\bd K)$.  We additionally set
\begin{equation*}
    \Omega_{\infty}^s(K) := \int_{\bd K} H^s_{d-1}(K,u) \, \vol_{\bd K}^s(\dint u).
\end{equation*}
For a geodesic ball we have
\begin{equation}\label{eqn:ball_p_float}
    \Omega_p^s(C_{\bo}(\alpha)) = \omega_{d-1} (\cos\alpha)^{\frac{p(d-1)}{d+p}} (\sin\alpha)^{\frac{d(d-1)}{d+p}},
\end{equation}
and $\Omega_\infty^s(C_{\bo}(\alpha)) = P^s(C_{\bo}(\alpha)^*) = P^s(C_{\bo}(\frac{\pi}{2}-\alpha)) = \omega_{d-1} (\cos\alpha)^{d-1}$.

Notice that $\Omega_{\infty}^s(K)$ is, up to normalization, the total measure of the absolutely continuous part $C_0^a(K,\cdot)$ of the $0$th-curvature measure $C_0(K,\cdot)$, see \cite[p.\ 256]{SW:2008}, i.e.,
\begin{equation}\label{eqn:polar_perimter}
    \frac{\Omega_{\infty}^s(K)}{2\omega_{d-1}} = C_0^a(K,\S^d) \leq C_0(K,\S^d) =\frac{P^s(K^*)}{2\omega_{d-1}}, 
\end{equation}
where $P^s(K)=\vol_{\bd K}^s(\bd K)$ is the perimeter of $K$, i.e., $(d-1)$-dimensional Hausdorff measure of $\bd K$ and where we used the duality relation, see \cite[Eq.\ (6.51)]{SW:2008}.
Moreover, if $K$ admits a rolling ball, or equivalently if $K$ is of class $\cC^{1,1}$, then $\Omega_{\infty}^s(K) = P^s(K^*)$, see \cite[Sec.~2.1]{BW:2024} for the Euclidean case, which transfers to the spherical case by considering the gnomonic projection of $K$.

\begin{remark}[$L_p$-floating area for $p=1$ and $p=-d/(d+2)$]
    For $p=1$ the $L_p$-floating area is the classical floating area and for $p=-d/(d+2)$ this curvature measure was recently derived in \cite[Thm.\ 1.1]{BW:2024} as the volume derivative of the spherical floating body conjugated by spherical duality. 
\end{remark}

The following generalizes a  duality formula for the $L_p$-affine surface area $\as_p$ which was established by Hug for $p>0$, see \cite[Thm.\ 3.2]{Hug2:1996} and, with different methods,  for $p \in [-\infty, 0)$ by Werner \& Ye \cite{WY:2008}. See also Ludwig \cite{Ludwig:2010} for duality for Orlicz affine surface areas.

\begin{theorem}[Duality Relation] \label{thm:duality}
Let  $d\geq 2$, $p\in\R\setminus\{-d\}$, and let $K\subset\S^d$ be a proper spherical convex body. If $p< 0$, then we additionally assume that $K$ is of class $\cC^2_+$ and if $p\in\{0,+\infty\}$ we assume that $K$ is of class $\cC^{1,1}$. Then
\begin{equation*}   
    \Omega_p^s(K^*) = \Omega_{d^2/p}^s(K).
\end{equation*}
\end{theorem}
\begin{proof}
    This was proven for $p=-d/(d+2)$ in \cite[Thm.\ 1.2 (iii)]{BW:2024} and the proof for general $p$ is analogous. Heuristically speaking, we use the transformation $\sigma_K := -\bn_K : \bd K \to \bd K^*$ which is a bi-Lipschitz if $K$ is of class $\cC^2_+$. For $p>0$ and general $K$ we may restrict to the boundary points where $H_{d-1}^s(K,\bx)>0$ and use the tools developed by Hug \cite[Sec.\ 2]{Hug1:1996} to use the transformation $\sigma_K$ restricted to this subset. In either case one may determine that the (approximate) Jacobian exits for almost every boundary point and is given by $J^{\bd K}(-\bn_K)(\bx) = H_{d-1}^s(K,\bx)$, which yields
    \begin{align*}
        \Omega_p^s(K^*) 
        &= \int_{\bd K^*} H_{d-1}^s(K^*,\bu)^{\frac{p}{d+p}} \, \vol_{\bd K^*}^s(\dint\bu)\\
        &= \int_{\bd K} H_{d-1}^s(K^*,\bn_K(\bx))^{\frac{p}{d+p}} H_{d-1}^s(K,\bx)\,\vol_{\bd K}^s(\dint\bu)\\
        &= \int_{\bd K} H_{d-1}^s(K,\bx)^{\frac{d}{d+p}}\,\vol_{\bd K}^s(\dint\bu) = \Omega_{d^2/p}^s(K),
    \end{align*}
    where in the last equality we used the fact that for almost all $\bx\in\bd K$ we have
    \begin{equation*}
        H_{d-1}^s(K^*,\bn_K(\bx)) H_{d-1}^s(K,\bx) = 1,
    \end{equation*}
    see, for example, \cite[Thm.\ 4.12]{BW:2024}.
\end{proof}
The next theorem describes inequalities and monotonicity behavior for the $L_p$-floating area similar to  the Euclidean $L_p$-affine surface area which 
were proved in \cite{WY:2008}.

\begin{theorem}\label{thm:sphere_inequalities}
    Let $d\geq 2$ and $p,q,r\in\mathbb{R}\setminus\{-d\}$. Let $K\subset \S^d$ be a proper spherical convex body and if $\min\{p,q,r\}<0$ additionally assume that $K$ is of class $\cC^2_+$. 
    If $t:= \frac{(q-r)(d+p)}{(p-r)(d+q)}>1$ and $t':=\frac{t}{t-1}$, then
    \begin{equation*}
        \Omega_p^s(K) \leq \Omega_q^s(K)^{1/t}\, \Omega_r^s(K)^{1/t'}.
    \end{equation*}
    Equality holds if $K$ is a geodesic ball. Moreover, if $K$ is of class $\cC^{1,1}$ and we have equality, then $K$ is a geodesic ball.

    For $r=0$, we further derive that
    \begin{equation}\label{eqn:monotone_inc}
        \left(\frac{\Omega_p^s(K)}{P^s(K)}\right)^{1+\frac{d}{p}} \leq \left(\frac{\Omega_q^s(K)}{P^s(K)}\right)^{1+\frac{d}{q}},
    \end{equation}
    for all $p,q$ such that $0<p<q$ or $p<q<-d$ or $p>0, q<-d$,
    With the same equality conditions as above.
 
    Furthermore, for $r\to \infty$, we find that
    \begin{equation}\label{eqn:monotone}
        \left(\frac{\Omega_p^s(K)}{P^s(K^*)}\right)^{1+\frac{p}{d}} \quad \text{is monotone decreasing for $p\geq 0$.}
    \end{equation}
    
    Finally, for $q\to \infty$ and $r=0$, we derive for all $p>0$ that
    \begin{equation}\label{eqn:p-iso_sphere}
        \Omega_p^s(K) \leq P^s(K)^{\frac{d}{d+p}} P^s(K^*)^{\frac{p}{d+p}}.
    \end{equation}
    and for $p<0$, $p\neq -d$, the inequality \eqref{eqn:p-iso_sphere} is reversed. Equality holds if and only if $K$ is a geodesic ball.
\end{theorem}
\begin{proof}\,
 We set $t:= \frac{(q-r)(d+p)}{(p-r)(d+q)}$ and $t':= \frac{t}{t-1} = \frac{(q-r)(d+p)}{(q-p)(d+r)}$. Since $\frac{1}{t}+\frac{1}{t'}=1$, we may use Hölder's inequality to obtain
  \begin{align*}
        \Omega_p^s(K) 
        &= \int_{\bd K} H_{d-1}^s(K,\bu)^{\frac{p}{d+p}}\, \Vol_{\bd K}^s(\dint \bu) \\
        &= \int_{\bd K} \left(H_{d-1}^s(K,\bu)^{\frac{r}{d+r}}\right)^{1/t'} \left(H_{d-1}^s(K,\bu)^{\frac{q}{d+q}}\right)^{1/t} \, \Vol_{\bd K}^s(\dint \bu) \\
        &\leq \left( \int_{\bd K} H_{d-1}^s(K,\bu)^{\frac{r}{d+r}}\, \Vol_{\bd K}^s(\dint u) \right)^{1/t'}  
            \left( \int_{\bd K} H_{d-1}^s(K,\bu)^{\frac{q}{d+q}}\, \Vol_{\bd K}^s(\dint \bu) \right)^{1/t}\\
        &= \Omega_r^s(K)^{1/t'} \,  \Omega_q^s(K)^{1/t}.
\end{align*}
If $K=C(\alpha)$ is a geodesic ball, then $H_{d-1}^s(K,\cdot)\equiv(\tan\alpha)^{-(d-1)}$ and we have equality.
On the other hand, if $K$ is of class $\cC^{1,1}$ and we have equality, then by the equality conditions of Hölder's inequality we conclude that $H_{d-1}^s(K,\bu)=\eta>0$ for almost every $\bu\in\bd K$. Since $K$ is of class $\cC^{1,1}$ the curvature measures $C_0(K,\cdot)$ is absolutely continuous with respect to $\Vol_{\bd K}$ and for $A\subset \bd K$ Borel we have
\begin{equation*}
    C_0(K,A) = \int_{A} H_{d-1}^s(K,\bx) \,\vol_{\bd K}^s(\dint \bx) = \eta \vol_{\bd K}^s(A).
\end{equation*}
Then by a result of Kohlmann \cite[Thm.\ 2]{Kohlmann:1998} we may conclude that $K$ is a geodesic ball.

Using the duality relations, see Theorem \ref{thm:duality}, we derive \eqref{eqn:monotone} from \eqref{eqn:monotone_inc}.

For the equality case in \eqref{eqn:p-iso_sphere} we do not need to a priori assume that $K$ is of class $\cC^{1,1}$, since the equality case also implies that
$\Omega_{\infty}^s(K)=P^s(K^*)$, which yields that $C_0(K,\cdot)$ is a absolutely continuous with respect to $\vol_{\bd K}^s$, see \eqref{eqn:polar_perimter} and compare also \cite[Thm.\ 7.6]{BW:2016}. Then as before we may conclude that $K$ is a geodesic ball.
\end{proof}

We first consider now an upper bound of $\Omega_p^s$ in terms of the volume radius $\alpha_K$, i.e, we want to find conditions such that $\Omega_p^s(K)\leq \Omega_p^s(C_K)$. Observe that such an upper bound can not hold true without some additional conditions on $K$ for $p\to 0^+$, since by the isoperimetric inequality we have for all all convex bodies $K\subset \S^d$, 
\begin{equation*}
    \Omega_0^s(K) = P^s(K) \geq P^s(C_K)=\Omega_0^s(C_K).
\end{equation*}

\begin{corollary}[Upper Bound by the Volume]\label{cor:p-float_ineq}
    Let $d\geq 3$ and let $K\subset \S^d$ be a proper spherical convex body of class $\cC^2_+$ and such that with respect to the GHS-center $\bo$ we have $\overline{K}=g_{\bo}(K)\subset \sqrt{d(d-2)} B_2^d$. Then for $p\geq 1$
    \begin{equation*}
        \Omega_p^s(K) \leq \Omega_p^s(C_K).
    \end{equation*}
    with equality if and only if $K$ is a geodesic ball.
    Furthermore, for $0< q\leq d^2$ and $p\geq q$ if $\overline{K}\supset \sqrt{d/q} B_2^d$, then
    \begin{equation*}
        \Omega_p^s(K)\leq \Omega_p^s(C_K).
    \end{equation*}
\end{corollary}
\begin{proof}
    Since the inequality is true for $p=1$, see Theorem \ref{thm:lambda_FI}, the statement follows by \eqref{eqn:monotone}, since
    \begin{align*}
        \frac{\Omega_p^s(K)}{P^s(K^*)} 
        \leq \left(\frac{\Omega_1^s(K)}{P^s(K^*)}\right)^{\frac{d+1}{d+p}}
        \leq \left(\frac{\Omega_1^s(C_K)}{P^s(K^*)}\right)^{\frac{d+1}{d+p}}
    \end{align*}
    and by the dual isoperimetric inequality, see Proposition \ref{cor:polar_iso}, $P^s(K^*) \leq P^s(C_K^*)$ we conclude
    \begin{align*}
        \Omega_p^s(K) \leq \Omega_1^s(C_K)^{\frac{d+1}{d+p}} P^s(K^*)^{\frac{p-1}{d+p}} \leq \Omega_1^s(C_K)^{\frac{d+1}{d+p}}P^s(C_K^*)^{\frac{p-1}{d+p}} =\Omega_p^s(C_K),
    \end{align*}
    where we use \eqref{eqn:ball_p_float} for the last equality.

    For the second statement let $q\leq d^2$ we may use the duality relation and conclude
    \begin{align*}
        \Omega_{q}^s(K) = \Omega^s_{d^2/q}(K^*) \leq \Omega^s_{d^2/q}(C_{K^*}) = \Omega^s_{q}(C_{K^*}^*),
    \end{align*}
    by Theorem \ref{thm:lambda_FI}, since $d^2/q\geq 1$ and 
    \begin{equation*}
        \overline{K}=g_{\bo}(K) = - g_{\bo}(K^*)^\circ \supset (\sqrt{d(d-2)}B_2^d)^\circ = (1/\sqrt{d(d-2)}) B_2^d.
    \end{equation*}
    By the spherical Blaschke--Santaló inequality, see Theorem \eqref{thm:BS}, we further conclude 
    \begin{equation*}
        C_{K^*}^* = C\left(\frac{\pi}{2}-\alpha_{K^*}\right) \supset C(\alpha_K)=C_K.
    \end{equation*}
    Since $\Omega^s_{q}(C(\alpha))$ is monotone decreasing for $\tan \alpha\geq \sqrt{d/q}\geq 1/\sqrt{d(d-2)}$, we derive
    \begin{equation*}
        \Omega_{q}^s(K) \leq \Omega^s_{q}(C_K),
    \end{equation*}
    for $\overline{K}\supset \sqrt{d/q} B_2^d$. The statement then follows for all $p\geq q$ by \eqref{eqn:monotone} as before.
\end{proof}

By the spherical isoperimetric inequality, see \eqref{eqn:iso_rad}, we have that $\alpha_P(K)\geq \alpha_K$, where we recall that $\alpha_P(K)$ is the radius such that 
\begin{equation}
    P^s(C(\alpha_P(K)))=P^s(K).
\end{equation}
Since, by \eqref{eqn:ball_p_float},
\begin{equation*}
    \frac{\Omega_p^s(C(\alpha))}{\omega_{d-1}} = (\cos \alpha)^{\frac{p(d-1)}{d+p}} (\sin\alpha)^{\frac{d(d-1)}{d+p}} = \frac{(\tan \alpha)^{\frac{d(d-1)}{d+p}}}{(1+(\tan\alpha)^2)^{\frac{d-1}{2}}}
\end{equation*}
is decreasing in $\alpha$ for $\tan\alpha \geq \sqrt{d/p}$ we have that $\Omega_p^s(C_K)\geq \Omega_p^s(C(\alpha_P(K)))$ if $\tan \alpha_K \geq \sqrt{d/p}$. Hence $C(\alpha_P(K))$ gives a better bound than $C_K$ in this case. At least if $d$ is odd, we can verify that $C(\alpha_P(K))$ is also an isoperimetric upper bound for $\Omega_p^s$.

\begin{theorem}[Upper Bound by the Perimeter in Odd Dimensions]\label{thm:improved_upper_bound}
    Let $d\geq3$ be odd, $p> 0$ and let $K\subset \S^d$ be a spherical convex body of class $\cC^2_+$. Then
    \begin{equation*}
        \Omega_p^s(K) \leq \Omega_p^s(C(\alpha_P(K))) 
        = P^s(K)^{\frac{d}{d+p}} P^s(C(\alpha_P(K))^*)^{\frac{p}{d+p}}
        = P^s(K)^{\frac{d}{d+p}} \left(\omega_{d-1}^{\frac{2}{d-1}}-P^s(K)^{\frac{2}{d-1}}\right)^{\frac{d-1}{2}\,\frac{p}{d+p}},
    \end{equation*}
    with equality if and only if $K$ is a geodesic ball, i.e., $K=C(\alpha_P(K))$.

    Moreover, if $\tan\alpha_K \geq \sqrt{d/p}$,
    then
    \begin{equation*}
        \Omega_p^s(K) \leq \Omega_p^s(C_K),
    \end{equation*}
    with equality if and only if $K$ is a geodesic ball.
\end{theorem}
\begin{proof}
    From \eqref{eqn:p-iso_sphere} we find
    \begin{equation*}
        \Omega_p^s(K) \leq P^s(K)^{\frac{d}{d+p}} P^s(K^*)^{\frac{p}{d+p}} = \left(\frac{P^s(K^*)}{P^s(C(\alpha_P(K))^*)}\right)^{\frac{p}{d+p}} \Omega_p^s(C(\alpha_P(K))).
    \end{equation*}
    Since $d$ is odd, the statement follows by the dual perimeter inequality (dPI), see Proposition \ref{cor:ppi}.
\end{proof}

\begin{remark}[2-dimensional case]
    Note that the dual perimeter inequality (dPI), see Proposition \ref{cor:ppi}, is reversed for $d=2$, since we have by the isoperimetric inequality and dual isoperimetric inequality (dII), see Proposition \ref{cor:polar_iso}, that
    \begin{equation*}
        P^s(K^*) = P^s(C_K^*) = P^s(C(\alpha_K)^*) \geq P^s(C(\alpha_P(K))^*),
    \end{equation*}
    which is equivalent to
    \begin{equation*}
        \Omega_p^s(C(\alpha_P(K))) \leq P^s(K)^{\frac{2}{2+p}} P^s(K^*)^{\frac{p}{2+p}}.
    \end{equation*}
    Furthermore
    \begin{equation*}
        \Omega_p^s(C_K) = P^s(C_K)^{\frac{2}{2+p}} P^s(C_K^*)^{\frac{p}{2+p}} \leq P^s(K)^{\frac{2}{2+p}} P^s(K^*)^{\frac{p}{2+p}}.
    \end{equation*}
    We also note that $\vol_2^s(K) = 2\pi(1-\cos\alpha_K)$, which yields
    \begin{equation*}
        \Omega_p^s(C_K) 
        = \vol_2^s(K)^{\frac{1}{p+2}} [2\pi - \vol_2^s(K)]^{\frac{p}{p+2}} [4\pi - \vol_2^s(K)]^{\frac{1}{p+2}},
    \end{equation*}
    and $P^s(K) = 2\pi \sin\alpha_P(K)$ gives
    \begin{equation*}
        \Omega_p^s(C(\alpha_P(K))) = P^s(K)^{\frac{2}{p+2}} [4\pi^2 - P^s(K)^2]^{\frac{p}{2(p+2)}}.
    \end{equation*}
    By the monotonicity of $\Omega^s_p(C_{\bo}(\alpha))$, we have
    \begin{equation*}
        \Omega_p^s(C_K)\leq \Omega_p^s(C(\alpha_P(K))) \leq P^s(K)^{\frac{2}{p+2}} P^s(K^*)^{\frac{p}{2+p}} \qquad \text{if $\tan\alpha_P(K)\leq \sqrt{2/p}$},
    \end{equation*}
    and
    \begin{equation*}
        \Omega_p^s(C(\alpha_P(K)))\leq \Omega_p^s(C_K) \leq P^s(K)^{\frac{2}{p+2}} P^s(K^*)^{\frac{p}{2+p}} \qquad \text{if $\tan\alpha_K\geq \sqrt{2/p}$}.
    \end{equation*}
\end{remark}

\begin{remark}[Is there a better bound for $\Omega_p^s$?]
    We have seen bounds of $\Omega_p^s$ in terms of the volume by $\Omega_p^s(C_K)$ and in terms of the perimeter $\Omega_p^s(C(\alpha_P(K))$. Depending on $K$, either bound could be better, and one might wonder if there is some natural bound smaller than both expressions. For example, we might consider
    \begin{equation*}
        P^s(C_K)^{\frac{d}{p+d}} P^s(C_{K^*})^{\frac{p}{p+d}} \leq \min\{\Omega_p^s(C_K),\Omega_p^s(C(\alpha_P(K))), P^s(K)^{\frac{d}{p+d}} P^s(K^*)^{\frac{p}{d+p}} \}.
    \end{equation*}
    Then a possible stronger form of the isoperimetric inequality for $\Omega_p^s$ would be
    \begin{equation}\label{eqn:strong_ineq}
        \Omega_p^s(K)\overset{?}{\leq} P^s(C_K)^{\frac{d}{d+p}} P^s(C_{K^*})^{\frac{p}{d+p}}.
    \end{equation}
    Note that \eqref{eqn:strong_ineq} holds true for $p\to +\infty$ asymptotically, since for $K$ of class $\cC^2_+$ we have $\Omega_\infty^s(K)= P^s(K^*)\leq P^s(C_K^*)$ by the dual isoperimetric inequality, see Proposition \ref{cor:polar_iso}. 
    If \eqref{eqn:strong_ineq} is true for some $p>0$, then it also holds for all $q\geq p$. This follows since by \eqref{eqn:monotone} and the dual isoperimetric inequality (dII), we find
    \begin{equation*}
        \Omega_q^s(K)
        \leq P^s(K^*) \left(\frac{\Omega_p^s(K)}{P^s(K^*)}\right)^{\frac{d+p}{d+q}} 
        \overset{\eqref{eqn:strong_ineq}}{\leq} P^s(C_K)^{\frac{d}{d+q}} P^s(C_K^*)^{\frac{p}{d+q}} P^s(K^*)^{\frac{q-p}{d+q}} \leq P^s(C_K)^{\frac{d}{d+q}} P^s(C_K^*)^{\frac{q}{d+q}}.
    \end{equation*}
    It therefore seems natural to ask if there is a critical value $q\in(0,+\infty)$ such that \eqref{eqn:strong_ineq} holds true for all $p\geq q$ and all proper spherical convex bodies $K\subset\S^d$.
\end{remark}

\section{Entropy of Euclidean convex bodies}\label{sec:entropy}

In this section we recall several notions of entropy and show how they are related to our investigations. For instance, a notion of 
entropy on convex bodies appears in the study of flows and evolution of hypersurfaces, see e.g., the  works by  Andrews \cite{Andrews:1997, Andrews:1999}, 
 Chow \cite{Chow:1991}, Hamilton \cite{Hamiltion:1994}, Daskalopoulos \& Lee \cite{Daskalopoulos:2004}, and   Guan \& Ni \cite{GN:2017}. 

Another notion of entropy was motivated by the work \cite{PW:2012}, where for a convex body $K$ in $\mathbb{R}^d$ of class $\cC^{1,1}$ that contains the origin in the interior, a centro-affine invariant \emph{entropy power functional} $\cE(K)$ was introduced as
\begin{align*}
  -\log \cE(K) 
    &= \frac{1}{d\vol_d^e(K^\circ)} \int_{\bd K} \frac{H_{d-1}^e(K,\bx)}{(\bx \cdot\bn_K(\bx))^{d+1}} \left[\log \frac{H_{d-1}^e(K,\bx)}{(\bx \cdot \bn_K(\bx))^{d+1}} \right] 
     \bx\cdot \bn_K(\bx) \, \vol_{\bd K}^e(\dint\bx)\\
    &=  \frac{1}{\vol_d^e(K^\circ)} \int_{\bd K}  \kappa_0(K,\bx) \left[\log \kappa_0(K,\bx)\right]\, V_K(\dint\bx),
\end{align*}
where $\kappa_0(K,\bx)=\frac{H_{d-1}^e(K,\bx)}{(\bx \cdot \bn_K(\bx))^{d+1}}$ is the \emph{centro-affine curvature}  and $V_K(\dint\bx)=(1/d) (\bx \cdot \bn_K(\bx)) \, \vol_{\bd K}^e(\dint\bx)$ is the centro-affine invariant \emph{cone volume measure}, see \cite{Stancu:2012} for important centro-affine notions on convex bodies.  We then define the \emph{centro-affine entropy functional} $E_{PW}$ by
\begin{equation}\label{def:PW}
    E_{PW}(K)  :=   -\log \cE(K).
\end{equation}
For general $K$ without assumptions on the curvature the centro-affine entropy is defined by
\begin{equation}
    E_{PW}(K) = -\frac{1}{\vol_d^e(K^\circ)} \int_{\bd K^\circ} [\log \kappa_0(K^\circ,\bx)]\, V_{K^\circ}(\dint\bx),
\end{equation}
where we note that this is equivalent to \eqref{def:PW} if $K$ is of class $\cC^{1,1}$ by \cite[Lem.\ 2.4]{BW:2024}.
For a centered Euclidean ball of radius $r>0$ we have
\begin{equation}
    E_{PW}(rB_2^d) = -2d \log r.
\end{equation}
We have 
\begin{equation}\label{eqn:InfoIneq}
    E_{PW}(K)  \geq -\log \frac{\vol_d^e(K)}{\vol_d^e(K^\circ)}, 
\end{equation}
which is called \emph{Information Inequality} in \cite{PW:2012}. We note that this inequality implies that
\begin{equation*}
    E_{PW}(K)+E_{PW}(K^\circ) \geq 0.
\end{equation*}
Moreover, if $\vol_d^e(K^\circ)\geq \kappa_d$ and $K$ is centered, then by the Blaschke--Santaló inequality, $E_{PW}(K)\geq 0$.

\begin{remark}[Conjectured centro-affine entropy inequality]
    As in the spherical space, see Remark \ref{rmk:conjecture_entropy}, we conjecture that for the entropy power functional we have
    \begin{equation}\label{eqn:centro_entropy_conj}
        \cE(K) \overset{?}\leq \cE(B_K),
    \end{equation}
    where $B_K$ is a centered Euclidean ball with the same volume as $K$. 
    Note that the this would imply that 
    \begin{equation}
        E_{PW}(K) \overset{?}{\geq} E_{PW}(B_K) = - 2 \log \frac{\vol_d^e(K)}{\kappa_d} \geq - \log \frac{\vol_d^e(K)}{\vol_d^e(K^\circ)},
    \end{equation}
    where the second inequality follows by the Blaschke--Santaló inequality. 
    Thus, the conjectured entropy inequality would improve the bound \eqref{eqn:InfoIneq}. In particular, it would follow that if $\vol_d^e(K)\leq\kappa_d$, then $E_{PW}(K)\geq 0$.
\end{remark}

In \cite{Chow:1991}, Chow defines the Gaussian entropy functional $E_C(K)$ for a convex body $K$ in $\mathbb{R}^d$ of class $\cC^2_+$, by
\begin{align*}
    E_C(K) 
    &= \frac{1}{\omega_{d-1}} \int_{\bd K} H_{d-1}^e(K,\bx) [\log H_{d-1}^e(K,\bx)] \, \vol_{\bd K}^e(\dint\bx)\\
    &= -\frac{1}{\omega_{d-1}} \int_{\S^{d-1}} [\log f_K(\bu)] \, \dint\bu, 
\end{align*}
where $f_K(\bu)$ is the curvature function of $K$, i.e., the reciprocal of the Gauss curvature at that point $\bx \in \bd K$ that has $\bu$ as outer normal.
By Jensen's inequality and the affine isoperimetric inequality (\ref{asa-inequality}), 
\begin{align}
    E_C(K) &\geq 
     - \frac{d+1}{d} \log \left(\frac{\as_1(K)}{\omega_{d-1}}\right)
    \geq 
    -\frac{d-1}{d} \log \frac{\vol_d^e(K)}{\kappa_d}.\label{eqn:CHOW}
\end{align}
So if $\vol_d^e(K)\leq \kappa_d=\omega_{d-1}/d$, then $E_C(K)\geq 0$.
Moreover, if $K$ is centered, then
\begin{equation*}
    E_C(K)+E_C(K^\circ) \geq -\frac{d-1}{d} \log\left(\frac{\vol_d^e(K)\vol_d^e(K^\circ)}{\kappa_d^2}\right) \geq 0.
\end{equation*}

Yet another entropy functional for convex bodies in $\mathbb{R}^d$ was defined by Andrews in \cite{Andrews:1997}, 
extending to general convex bodies  a definition given by Fiery \cite{Fiery:1977}  in the  symmetric case. Subsequently it was used by
Guan \& Ni \cite{GN:2017}.  It is defined as follows, 
\begin{align*}
    E_{h}(K) 
    &= \max_{\bz\in K} \frac{1}{\omega_{d-1}} \int_{\S^{d-1}} [\log h_{K-\bz}(\bu)] \, \dint \bu\\
    &= \max_{\bz\in K} \frac{1}{\omega_{d-1}} \int_{\bd K} H_{d-1}^e(K,\bx) \log (( \bx-\bz) \cdot \bn_K(\bx)) \, \vol_{\bd K}^e(\dint \bx).
\end{align*}
It was observed in [Prop.\ 1.1, Guan \& Ni] that  $E_h(K) \geq 0$. In general we have
\begin{equation}\label{eqn:GNIneq}
    E_h(K) \geq \frac{1}{d} \log \frac{\vol_d^e(K)}{\kappa_d},
\end{equation}
and thus
\begin{equation*}
    E_h(K)+E_h(K^\circ) \geq \frac{1}{d} \log\left(\frac{\vol_d^e(K)\vol_d^e(K^\circ)}{\kappa_d^2}\right).
\end{equation*}
In [Prop.\ 3.1, Guan \& Ni] it is shown that 
\begin{equation}\label{eqn:GN}
    E_C(K) \geq E_h(K) - \log \frac{\vol_d^e(K)}{\kappa_d} \geq -\frac{d-1}{d} \log\frac{\vol_d^e(K)}{\kappa_d}.
\end{equation}
Thus \eqref{eqn:GN} improves the bound \eqref{eqn:CHOW} and exhibits a relation between $E_C$ and $E_h$.  Moreover, the entropies $E_{PW}$, $E_C$ and $E_h$ are related as follows
\begin{align*}
    \frac{1}{\cE(K)}\geq \frac{\vol_d^e(K^\circ)}{\vol_d^e(K)} 
    &= \frac{1}{d\vol_d^e(K)} \int_{\bd K} \frac{H_{d-1}^e(K,\bx)}{(\bx \cdot \bn_{K}(\bx))^{d}}
         \, \vol_{\bd K}(\dint\bx)\\
    &\geq \frac{\kappa_d}{\vol_d^e(K)} \exp\left( \frac{-1}{\kappa_d}\int_{\bd K} [\log \bx \cdot \bn_{K}(\bx)] H_{d-1}^e(K,\bx) \, \vol_{\bd K}^e(\dint\bx)\right)\\
    & \geq \frac{\kappa_d}{\vol_d^e(K)} \exp[-d E_h(K)].
\end{align*}
Therefore
\begin{equation*}
    E_{PW}(K) \geq -dE_h(K) - \log \frac{\vol_d^e(K)}{\kappa_d} \geq -dE_C(K) - (d+1) \log\frac{\vol_d^e(K)}{\kappa_d}.
\end{equation*}
Thus $E_C(K)$ and $E_h(K)$ give lower bounds (where $E_h$ gives a better bound) for $E_{PW}(K)$, but the conjectured inequality \eqref{eqn:centro_entropy_conj} would be a better, and potentially optimal, lower bound, since by \eqref{eqn:GNIneq}
\begin{equation*}
    E_{PW}(K) \overset{?}\geq -2 \log \frac{\vol_d^e(K)}{\kappa_d} \geq -dE_h(K) - \log\frac{\vol_d^e(K)}{\kappa_d}.
\end{equation*}

\section{Entropy of spherical convex bodies}\label{sec:entropy_sphere}

Entropy functionals are important  in many areas of mathematics and applications, for instance in  probability theory, information theory, and many more.
We will show in this section that a notion of entropy naturally also appears in our context.

Let $d\geq 2$ and let $K\subset \S^d$ be a proper spherical convex body. 
By Theorem \ref{thm:sphere_inequalities} we find that $\left(\frac{\Omega_q^s(K^*)}{P^s(K^*)}\right)^{1+\frac{d}{q}}$ 
is monotone increasing for $q\geq 0$. We therefore define
\begin{equation}\label{def:entropy}
    \cE^s(K) := \lim_{q\to 0^+} \left(\frac{\Omega_q^s(K^*)}{P^s(K^*)}\right)^{1+\frac{d}{q}}.
\end{equation}
Note that if $K$ is a spherical polytope, then $\cE^s(K)=0$. Furthermore, for all $q\geq 0$,
\begin{equation}\label{eqn:entropy_bound}
    \cE^s(K) \leq \left(\frac{\Omega_q^s(K^*)}{P^s(K^*)}\right)^{1+\frac{d}{q}} \leq \frac{P^s(K)}{P^s(K^*)}.
\end{equation}
If $K$ is of class $\cC^{1,1}$, then $P^s(K^*)=\Omega_{\infty}^s(K)$ and by Theorem \ref{thm:duality} we also have that
\begin{equation*}
    \cE^s(K) = \lim_{p\to\infty} \left(\frac{\Omega_p^s(K)}{\Omega_\infty^s(K)}\right)^{1+\frac{p}{d}}.
\end{equation*}
For a spherical cap $C(\alpha)$, with $\alpha\in (0,\frac{\pi}{2})$, we have
\begin{equation*}
    \cE^s(C(\alpha)) 
    = \left(\frac{\Omega_p^s(C(\alpha))}{P^s(C(\alpha)^*)}\right)^{1+\frac{p}{d}} 
    = \frac{P^s(C(\alpha))}{P^s(C(\alpha)^*)} 
    = (\tan\alpha)^{d-1},
\end{equation*}
for all $p\geq 0$.

We call $\cE^s(K)$ the \emph{(spherical) entropy power functional}. This is motivated  by the following representation.
\begin{theorem}[Spherical Curvature Entropy]\label{thm:KL}
    Let $d\geq 2$ and let $K\subset \S^d$ be a proper spherical convex body. Then
    \begin{equation*}
        E^s(K) := -\log \cE^s(K) = -\frac{1}{P^s(K^*)} \int_{\bd K^*} \log H_{d-1}^s(K^*,\bu) \, \vol_{\bd K^*}^s(\dint\bu).
    \end{equation*}
    If $K$ is of class $\cC^{1,1}$, then also
    \begin{equation}\label{eqn:entropy_curvature}
        E^s(K)= \frac{1}{P^s(K^*)} \int_{\bd K} H_{d-1}^s(K,\bu) \left[\log H_{d-1}^s(K,\bu)\right] \, \Vol_{\bd K}^s(\dint \bu). 
    \end{equation}
    We call $E^s(K)$ the \emph{spherical curvature entropy} of $K$.
\end{theorem}
\begin{proof}
    By \eqref{def:entropy} and L'Hospital's rule we calculate
    \begin{align*}
        \log \cE^s(K) 
        &= \lim_{q\to 0^+} \left(1+\frac{d}{q}\right) \log \frac{\Omega_q^s(K^*)}{P^s(K^*)}
        = \lim_{q\to 0^+} \frac{(q+d)^2}{d} \frac{1}{\Omega_q^s(K^*)} \frac{\dint}{\dint q} \Omega_q^s(K^*)\\
        &= \frac{1}{P^s(K^*)} \int_{\bd K^*} \left(\lim_{q\to 0^+} H_{d-1}^s(K^*,\bu)^{\frac{q}{d+q}} [\log H_{d-1}^s(K^*,\bu)]\right) \, \Vol_{\bd K^*}^s(\dint\bu)\\
        &= \frac{1}{P^s(K^*)} \int_{\bd K^*} \log H_{d-1}^s(K^*,\bu)\, \vol_{\bd K^*}^s(\dint\bu).
    \end{align*}
    For the second statement we just note that since
    $K$ is of class $\cC^{1,1}$ we have that $\Omega_{\infty}^s(K)=P^s(K^*)$. Thus, by applying L'Hospital's rule on \eqref{def:entropy} we may calculate as above and derive \eqref{eqn:entropy_curvature}.
\end{proof}

\begin{remark}[Relative Entropy]
Let $(\mathbb{X}, \mu)$ be a measure space and $f$ and $g$ be probability densities on $\mathbb{X}$ with respect to $\mu$. The \emph{Kullback--Leibler divergence} or \emph{relative entropy} from $f$ to $g$ is defined as in \cite[Ch.~8]{CT:2006}
\begin{equation}
	D_{KL}(f||g) = -\int_{\mathbb{X}} \log\left( \frac{g(\bx)}{f(\bx)} \right)\, f(\bx)\, \mu(\dint \bx).
\end{equation}
Let $K\subset \S^d$ be a proper spherical convex body of class $\cC^{1,1}$.
By Theorem \ref{thm:KL} we observe that for the measure space $(\bd K, \vol_{\bd K}^s)$
\begin{equation*}
    D_{KL}(f||g) = \log\left(\frac{P^s(K)}{P^s(K^*)}\right) -\log E^s(K),
\end{equation*}
where we set
\begin{equation}\label{P+Q}
    f  := \frac{H_{d-1}^s(K,\cdot)}{P^s(K^*)} \qquad \text{and}\qquad g  := \frac{1}{P^s(K)}.
\end{equation}
Note that \eqref{eqn:entropy_bound} is equivalent to \emph{Gibb's inequality} $D_{KL}(f||g)\geq 0$.

The relative entropy $D_{KL}(f||g)$ measures the gap between $\cE^s(K)$ and $\frac{P^s(K)}{P^s(K^*)}$ in the inequality \eqref{eqn:entropy_bound} and can be seen as  quantifying the information gain when replacing $g$ by $f$.
If $K$ is of class $\cC^{1,1}$, then the density $f$ is exactly the Radon--Nikodým derivative with respect to $\mu=\vol_{\bd K}^s$ of the probability measure $\nu$ on $\bd K$ defined by
\begin{equation*}
    \nu(A) := \frac{C_0(K, A)}{P^s(K^*)} = \frac{\vol_{\bd K^*}^s(\{\bu\in\bd K^*: \bn_{K^*}(\bu)\in A\})}{P^s(K^*)}, \qquad \text{for all Borel $A\subset \bd K$},
\end{equation*}
that is, $f=\frac{\dint\nu}{\dint\mu}$.

If $K\subset \S^d$ is a proper spherical polytope, then $\nu = \frac{C_0(K,\cdot)}{P^s(K^*)}$ is a discrete measure on $\bd K$ concentrated at the vertices $\mathrm{vert}\, K$ of $K$, that is,
\begin{align*}
    \nu(A) 
    = \sum_{\bu\in \mathrm{vert}\, K} \frac{\vol_{d-1}^s(F(K^*,\bu))}{P^s(K^*)} \mathbf{1}[\bu \in A], \qquad \text{for all Borel $A\subset \bd K$},
\end{align*}
where $F(K^*,\bu)$ is the face of $K^*$ with normal direction $\bu$.
In particular, we have $\cE^s(K)=0$ and $D_{KL}(\nu||g) = +\infty$.
\end{remark}

By \eqref{eqn:entropy_bound} we derive the following bound for the entropy functional.
\begin{corollary}[Information Inequality]\label{cor:entropy_bound}
    Let $d\geq 2$ and let $K\subset \S^d$ be a proper spherical convex body. Then
    \begin{equation*}
        \cE^s(K) \leq \frac{P^s(K)}{P^s(K^*)},
    \end{equation*}
    with equality if and only if $K$ is a geodesic ball. In particular, we have
    \begin{equation}\label{eqn:polar_entropy}
        \cE^s(K)\cE^s(K^*)\leq 1.
    \end{equation}
\end{corollary}

The isoperimetric and dual isoperimetric inequality (dII) imply the following dual entropy inequality.

\begin{theorem}[Dual Entropy Inequality]\label{thm:entropy_polar}
    Let $d\geq 2$ and let $K\subset \S^d$ be a proper spherical convex body. If $d\neq 2$, then we additionally assume that $K$ is of class $\cC^2_+$. Then
    \begin{equation*}
        \cE^s(K^*) \leq \cE^s(C_K^*),
    \end{equation*}
    with equality if and only if $K$ is a geodesic ball.
\end{theorem}
\begin{proof}
    By \eqref{eqn:entropy_bound}, the isoperimetric inequality and the dual isoperimetric inequality (dII), see Proposition \ref{cor:polar_iso}, we conclude
    \begin{equation*}
        \cE^s(K^*) \leq \frac{P^s(K^*)}{P^s(K)} \leq \frac{P^s(C_K^*)}{P^s(C_K)} = \cE^s(C_K^*). \qedhere
    \end{equation*}
\end{proof}

\begin{remark}[Conjectured Entropy Inequality]\label{rmk:conjecture_entropy}
    We may ask if there is an isoperimetric inequality for the entropy functional such that
    \begin{equation}\label{eqn:entropy_ineq}
        \cE^s(K) \overset{?}{\leq} \cE^s(C_K).
    \end{equation}
    By the isoperimetric and dual isoperimetric inequality (dII)
    \begin{equation*}
        \cE^s(C_K) = \frac{P^s(C_K)}{P^s(C_K^*)} \leq \frac{P^s(K)}{P^s(K^*)}.
    \end{equation*}
    Finally,
    by the spherical Blaschke--Santaló inequality, see Theorem \ref{thm:BS}, we have
    \begin{equation*}
        \cE^s(K^*) \overset{\eqref{eqn:entropy_ineq}}{\leq} \cE^s(C_{K^*}) = (\tan \alpha_{K^*})^{d-1} \leq (\tan \alpha_K)^{1-d} = \cE^s(C_K^*).
    \end{equation*}
    We see that \eqref{eqn:entropy_ineq} would imply Theorem \ref{thm:entropy_polar}.
\end{remark}

\begin{theorem}[(Strong) Floating Area Inequality implies Entropy Inequality]\label{cor:entropy_strong_bound}
    Let $d\geq 2$ and let $K\subset \S^d$ be a proper spherical convex body. If there is $p>0$ such that
    \begin{equation}\label{eqn:stronger_ineq}
        \frac{\Omega_p^s(K)}{\Omega_p^s(C_K)} \leq \frac{P^s(K^*)}{P^s(C_K^*)},
    \end{equation}
    then
    \begin{equation*}
        \cE^s(K)\leq \cE^s(C_K).
    \end{equation*}
    For $d\geq 3$ \eqref{eqn:stronger_ineq} together with the dual isoperimetric inequality (dII), see Proposition \ref{cor:polar_iso}, implies $\Omega_p^s(K)\leq \Omega_p^s(C_K)$ and for $d=2$ \eqref{eqn:stronger_ineq} is equivalent to $\Omega_p^s(K)\leq \Omega_p^s(C_K)$.
\end{theorem}
\begin{proof}
    By Corollary \ref{cor:entropy_bound} and \eqref{eqn:stronger_ineq}, we find
    \begin{equation*}
        \cE^s(K) \leq \left(\frac{\Omega_p^s(K)}{P^s(K^*)}\right)^{1+\frac{p}{d}} \leq \left(\frac{\Omega_p^s(C_K)}{P^s(C_K^*)}\right)^{1+\frac{p}{d}} = \frac{P^s(C_K)}{P^s(C_K^*)} = \cE^s(C_K).
        \qedhere
    \end{equation*}
\end{proof}

Basit, Hoehner, Lángi and Ledford \cite[Thm.~3.8]{BHLL:2024} have show that $\Omega^s_1(K)\leq \Omega^s_1(C_K)$ for $d=2$, and therefore we have the following
\begin{corollary}[Entropy Inequality on $\S^2$ for smooth symmetric convex bodies]
    Let $K\subset \S^{2}$ be a proper spherical convex body of class $\cC^2_+$ that is symmetric about a point $\bo\in\S^{2}$. Then
    \begin{equation*}
        \cE^s(K)\leq \cE^s(C_K).
    \end{equation*}
\end{corollary}

\begin{remark}
    Note that \eqref{eqn:strong_ineq} and \eqref{eqn:stronger_ineq} are connected as follows:
    if \eqref{eqn:strong_ineq} holds true and 
    \begin{equation}\label{eqn:connection1}
        P^s(C_{K^*})^{\frac{p}{d+p}} P^s(C_K^*)^{\frac{d}{d+p}} \leq P^s(K^*),
    \end{equation}
    then also \eqref{eqn:stronger_ineq} follows. Conversely, if \eqref{eqn:stronger_ineq} holds true and
    \begin{equation}\label{eqn:connection2}
        P^s(C_{K^*})^{\frac{p}{d+p}} P^s(C_K^*)^{\frac{d}{d+p}} \geq P^s(K^*),
    \end{equation}
    then \eqref{eqn:strong_ineq} follows. For $d=2$ the isoperimetric inequality (II) and equality in the dual isoperimetric inequality (dII) show that \eqref{eqn:connection1} holds true. However, for $d\geq 3$ we have by the isoperimetric and dual isoperimetric inequality that
    \begin{equation*}
        P^s(C_{K^*})\leq P^s(K^*)\leq P^s(C_K^*),
    \end{equation*}
    and \eqref{eqn:connection1} seems to hold asymptotically for $p\to \infty$ and \eqref{eqn:connection2} seems to hold asymptotically for $p\to 0^+$.
\end{remark}

\section{\texorpdfstring{$L_p$}{Lp}-floating area and Entropy and in real space forms}\label{sec:real_limit}

We now investigate real-analytic extension of the $p$-affine surface area and of our notion of entropy to real space forms of constant curvature $\lambda\geq 0$. A real-analytic extension for the affine surface area, that is, $p=1$, was obtained in \cite{BW:2018}. In \cite{BW:2024} we used the volume of the polar of the floating body to naturally derive $p$-floating area for $p_d:=-d/(d+2)$ on the sphere and hyperbolic space, and observed that there are two different real-analytic extensions: 
On the one hand, by fixing a point in all space forms an introducing a weight function, one can connect the centro-affine $p_d$-affine surface area in Euclidean space with the $p_d$-floating area in non-Euclidean spaces. This particular weight function was naturally derived in \cite[Thm.\ 4.17]{BW:2018} by considering a polarity operator relative to the fixed point.

On the other hand, when we do not fix a point, then we can connect the $p_d$-floating area to a rigid-motion invariant curvature measure in Euclidean space. Note that both extensions give the same quantity in the case $p=1$, since Blaschke's affine surface area is equi-affine invariant and therefore both, centro-affine invariant and rigid-motion invariant.

In the same way we can extend the spherical $p$-affine surface area for general $p\neq -d$ as well as the entropy. Remarkably, in case of the entropy, we find that one extension connects with the rigid-motion invariant Gaussian entropy measure $E_C$ in the Euclidean case, whereas the other case connects with the centro-affine invariant entropy measure $E_{PW}$.

\subsection{Real-analytic extension of \texorpdfstring{$\as_p$}{Lp-affine surface area} to \texorpdfstring{$\Omega_p^s$}{Lp-floating area} and \texorpdfstring{$E_{PW}$}{centro-affine entropy} to \texorpdfstring{$E^s$}{spherical curvature entropy}}

Let $\Sp^d(\lambda)$ be the real space form of dimension $d$ and constant curvature $\lambda \geq 0$.
Let $p\in\R$, $p\neq -d$, and let $K\subset \Sp^d(\lambda)$ be a convex body, where we assume that $K$ is of class $\cC^2_+$ if $p<0$, and such that $K$ contains $\bo\in\Sp^d(\lambda)$ in the interior and is contained in the open half-space centered at $\bo$.
We  define a real-analytic extension of the $L_p$-affine surface area by
\begin{equation}
    \as_p^{\lambda,\bo}(K) 
        := \int_{\bd K} \left(\frac{H^\lambda_{d-1}(K,\bu)}{f_{\bo}^\lambda(K,\bu)^{d+1}}\right)^{\frac{p}{d+p}} \, f_{\bo}^\lambda(K,\bu)\, \Vol_{\bd K}^\lambda(\dint \bu), \qquad \text{for $p\neq -d$}.
\end{equation}
Here the weight function $f^\lambda_{\bo}(K,\cdot)$, which was derived in \cite[Thm.\ 4.17]{BW:2024}, is defined by
\begin{equation*}
        f^\lambda_{\bo}(K,\bu) := \sqrt{\left|\frac{\lambda + (\tan_{\lambda} d_\lambda(\bo,H(K,\bu)))^2}{1+\lambda (\tan_{\lambda} d_\lambda(\bo,H(K,\bu)))^2}\right|}
    \end{equation*}
where $H(K,\bu)$ denotes the tangent hyperplane to $K$ at $\bu$, and $d_{\lambda}(\bo,H(K,\bu))$ is the minimal geodesic distance in $\Sp^d(\lambda)$ of $\bo$ to the points in $H(K,\bu)$.
In particular, 
\begin{equation*}
    \as_0^{\lambda,\bo}(K) = \int_{\bd K} f_\bo^\lambda(K,\bu)\, \Vol_{\bd K}^\lambda(\dint \bu), 
\end{equation*}
and
\begin{equation*}
  \as_\infty^{\lambda,\bo}(K)   =   \lim_{p\to \infty} \as_p^{\lambda,\bo}(K) 
        = \int_{\bd K} \frac{H^\lambda_{d-1}(K,\bu)}{f_{\bo}^\lambda(K,\bu)^d} \, \Vol_{\bd K}^\lambda(\dint \bu).
\end{equation*}
Note that for $\lambda =1$ we have $f^\lambda_{\bo} \equiv 1$, which yields $\as^{1,\bo}_p(K)=\Omega_p^s(K)$.

\begin{example}
        For $\lambda>0$ and a  geodesic ball $C^\lambda_{\bo}(\alpha)$ centered in $\bo\in\Sp^d(\lambda)$ of radius $\alpha\in[0,\frac{\pi}{2\sqrt{\lambda}}]$ we find that
        $H_{d-1}^\lambda(C^\lambda_{\bo}(\alpha),\cdot) \equiv (\tan_\lambda \alpha)^{-(d-1)}$ and 
        \begin{equation*}
            f_{\bo}^\lambda(C_\bo^\lambda(\alpha),\cdot) 
                \equiv \sqrt{\lambda (\cos_\lambda \alpha)^2 + (\sin_{\lambda} \alpha)^2},
        \end{equation*}
        where we recall \eqref{eqn:lambda_sin} for the definition of $\sin_{\lambda}$ and $\cos_\lambda$.
    This yields
    \begin{align*}
        \as_0^{\lambda, {\bo}}(C^\lambda_{\bo}(\alpha)) 
            &= \omega_{d-1} (\sin_\lambda \alpha)^{d-1} \sqrt{\lambda(\cos_\lambda \alpha)^2+(\sin_\lambda \alpha)^2}\\
        \as_\infty^{\lambda, {\bo}}(C^\lambda_{\bo}(\alpha)) 
            &= \omega_{d-1} (\cos_\lambda \alpha)^{d-1} (\lambda (\cos_\lambda \alpha)^2+(\sin_\lambda \alpha)^2)^{-d/2}
    \end{align*}
    and
    \begin{align*}
        \as_p^{\lambda, {\bo}}(C^\lambda_{\bo}(\alpha)) 
            &=  \omega_{d-1} (\sin_\lambda \alpha)^{\frac{d(d-1)}{d+p}} (\cos_\lambda\alpha)^{\frac{p(d-1)}{d+p}} 
                (\lambda (\cos_\lambda \alpha)^2+(\sin_\lambda \alpha)^2)^{-\frac{1}{2} \frac{d(p-1)}{d+p}}\\
            &= \as_0(C^\lambda_{\bo}(\alpha))^{\frac{d}{d+p}} \as_\infty(C^\lambda_{\bo}(\alpha))^{\frac{p}{d+p}},
    \end{align*}
\end{example}

A duality relation similar to the one of Theorem \ref{thm:duality} holds for the $\bo$-polarity in $\Sp^d(\lambda)$ which was introduced in \cite[Sec.\ 4.4]{BW:2024}. Note that for $\lambda=0$, $\bo$-polarity is just the usual polarity on Euclidean convex bodies, that is, $K^{\bo}=K^\circ$, and for $\lambda=1$, $\bo$-polarity is the reflection of the spherical duality, that is, $K^{\bo}=-K^*$. 
\begin{theorem} \label{Dual}
Let $d\geq 2$, $\lambda>0$, $p\in\R\setminus\{-d\}$, and
let $K\subset \Sp^d(\lambda)$ be a convex body that  contains $\bo\in\Sp^d(\lambda)$ in the interior and is contained in the open half-space centered in $\bo$. 
If $p\leq 0$ or $p=\infty$, then additionally assume that $K$ is of class $\cC^2_+$.
Then
\begin{equation*}
    \as_p^{\lambda,\bo}(K) = \as_{d^2/p}^{\lambda,\bo}(K^\bo),
\end{equation*}
where $K^\bo$ is the $\bo$-polar of $K$, see \cite[Sec.\ 4.4]{BW:2024}. 
In particular, if $K$ is of class $\cC^2_+$, then 
\begin{equation*}
    \as_\infty^{\lambda,\bo}(K) = \as_{0}^{\lambda,\bo}(K^\bo).
\end{equation*}
\end{theorem}
\begin{proof}
    We choose a projective model $\overline{K}\subset \R^d$ of $K\in\Sp^d(\lambda)$. Then by \cite[Eq.\ 4.11]{BW:2024} and \cite[Eq.\ (4.15)]{BW:2024}
    \begin{align*}
        H^\lambda_{d-1}(\overline{K},\bx) f_{\bo}^\lambda(\overline{K},\bx)^{-(d+1)} 
        &= H_{d-1}^e(\overline{K},\bx)\|\bx^\circ\|^{d+1} 
            \left(\frac{1+\lambda\|\bx\|^2}{1+\lambda \|\bx^\circ\|^2}\right)^{\frac{d+1}{2}}\\
        &=\kappa_0(\overline{K},\bx) \left(\frac{1+\lambda\|\bx\|^2}{1+\lambda \|\bx^\circ\|^2}\right)^{\frac{d+1}{2}},
    \end{align*}
    where $\bx^\circ = \frac{\bx}{\bx\cdot \bn_{\overline{K}}(\bx)}\in\bd \overline{K}^\circ$ is the polar point map and $\kappa_0(\overline{K},\cdot)$ is the centro-affine curvature function of $K$. In particular, 
    \begin{equation*}
        \frac{H^\lambda_{d-1}(\overline{K},\bx)}{f_{\bo}^\lambda(\overline{K},\bx)^{d+1}} \, \frac{H^\lambda_{d-1}(\overline{K}^\circ,\bx^\circ)}{f_{\bo}^\lambda(\overline{K}^\circ,\bx^\circ)^{d+1}} = \kappa_0(\overline{K},\bx) \, \kappa_0(\overline{K}^\circ,\bx^\circ)= 1,
    \end{equation*}
    for almost all $\bx\in\bd \overline{K}$, where we used that $(\bx^\circ)^\circ=\bx$.
    Using \cite[Lem.\ 2.4]{BW:2024} we conclude
    \begin{align*}
        \as_p^{\lambda,\bo}(\overline{K}) 
            &= \int_{\bd\overline{K}} \left(\frac{H^\lambda_{d-1}(\overline{K},\bx)}{f_{\bo}^\lambda(K,\bx)^{d+1}}\right)^{\frac{p}{d+p}}
                f_{\bo}^\lambda(\overline{K},\bx)\, \vol_{\bd \overline{K}}^\lambda(\dint \bx)\\
            &= d \int_{\bd\overline{K}} \kappa_0(\overline{K},\bx)^{\frac{p}{d+p}}
                \left(\frac{1+\lambda\|\bx\|^2}{1+\lambda\|\bx^\circ\|^2}\right)^{\frac{(d+1)p}{2(d+p)}} 
                    \frac{\sqrt{1+\lambda\|\bx^\circ\|^2}}{(1+\lambda\|\bx\|^2)^{\frac{d}{2}}}
                    \, V_{\overline{K}}(\dint\bx)\\
            &= d \int_{\bd\overline{K}^\circ} \kappa_0(\overline{K}^\circ,\by)^{\frac{d}{d+p}}
                \left(\frac{1+\lambda\|\by^\circ\|^2}{1+\lambda\|\by\|^2}\right)^{\frac{(d+1)p}{2(d+p)}} 
                    \frac{\sqrt{1+\lambda\|\by\|^2}}{(1+\lambda\|\by^\circ\|^2)^{\frac{d}{2}}} 
                    \, V_{\overline{K}^{\circ}}(\dint\by)\\
            &= d\int_{\bd \overline{K}^\circ} \kappa_0(\overline{K}^\circ,\by)^{\frac{d}{d+p}}
                \left(\frac{1+\lambda\|\by\|^2}{1+\lambda\|\by^\circ\|^2}\right)^{\frac{(d+1)d}{2(d+p)}}
                \frac{\sqrt{1+\lambda \|\by^\circ\|^2}}{(1+\lambda\|\by\|^2)^{\frac{d}{2}}} 
                \, V_{\overline{K}^\circ}(\dint\by)\\
            &= \int_{\bd \overline{K}} \left(\frac{H^\lambda_{d-1}(\overline{K}^\circ,\by)}{f_{\bo}^\lambda(K^\circ,\by)^{d+1}}\right)^{\frac{d}{d+p}}
                f_{\bo}^\lambda(\overline{K}^\circ,\by)\, \vol_{\bd \overline{K}^\circ}^{\lambda}(\dint\by)
            =\as_p^{\lambda,\bo}(\overline{K}^\circ).\qedhere
    \end{align*}
\end{proof}

For $K\subset \Sp^d(\lambda)$ such that $K$ contains $\bo$ in the interior is contained in the half-space around $\bo$, we can use the gnomonic projection $g^\lambda_{\bo} :\Sp^d(\lambda)\to \R^d$ to identify $K$ with the projective model $\overline{K}:=g^\lambda_{\bo}(K)\subset \R^d$. 

\begin{theorem}\label{thm:limit_center}
    Let $\overline{K}\subset\R^d$ be a convex body that contains the origin in the interior. 
    For $\lambda> 0$ and a fixed point $\bo\in\Sp^d(\lambda)$ we consider $\R^d$ as projective model of $\Sp^d(\lambda)$ around $\bo$. Then
    \begin{equation}\label{lambda=null}
        \lim_{\lambda\to 0^+} \as_p^{\lambda,\bo}(\overline{K}) 
         = \as_p(\overline{K})
    \end{equation}
    and
    \begin{align}
        \lim_{\lambda\to 1} \as_p^{\lambda,\bo}(\overline{K})
         = \Omega_p^s(\overline{K}).
    \end{align}
\end{theorem}
\begin{proof}
    By \cite[Eq.~(4.13)]{BW:2018} and \cite[Eq.~(4.15)]{BW:2024}, we find
    \begin{align*}
        \as_p^{\lambda,\bo}(\overline{K}) 
        &= \int_{\bd \overline{K}} \left( H_{d-1}^e(\overline{K},\bx) 
        \left(\frac{1+\lambda\|\bx\|_2^2}{\lambda +(\bx\cdot \bn_{\overline{K}}(\bx))^2}\right)^{\frac{d+1}{2}} \right)^{\frac{p}{d+p}}
        \frac{\sqrt{\lambda+(\bx\cdot \bn_{\overline{K}}(\bx))^2}}{(1+\lambda\|\bx\|_2^2)^{d/2}} \, \vol^e_{\bd \overline{K}}(\dint \bx).
    \end{align*}
    Taking the limit $\lambda\to 0^+$, respectively $\lambda\to 1$, the statement follows.
\end{proof}

\begin{remark}
Thus, for $\lambda =0$ we recover the usual $L_p$-affine surface area in $\mathbb{R}^d$ and for $\lambda =1$ we recover the $L_p$-floating  area
on $\mathbb{S}^d$. In particular, for $p=0$, we see that the spherical surface area $P^s(\overline{K})=\lim_{\lambda\to 1} \as^{\lambda,\bo}_0(\overline{K})$ can be related for to the Euclidean volume $d\vol_d^e(\overline{K}) = \lim_{\lambda\to 0^+} \as_0^{\lambda,\bo}(\overline{K})$.
\end{remark}

Theorem \ref{thm:sphere_inequalities} naturally extends to $\as_p^{\lambda,\bo}$ as follows.  
\begin{theorem}\label{new-lambda_inequalities}
    Let $d\geq 2$, $\lambda\geq 0$, and $p,q,r\in\R\setminus\{-d\}$.
    Let $K\subset \Sp^d(\lambda)$ be a convex body that contains $\bo\in\Sp^d(\lambda)$ in the interior and is contained in the open half-space centered in $\bo$. 
    Additionally assume that $K$ is of class $\cC^2_+$ if $\min\{p,q,r\}<0$.
    
    If $t:= \frac{(q-r)(d+p)}{(p-r)(d+q)}>1$ and $t':=\frac{t}{t-1}$, then
    \begin{equation*}\label{lambda-inequality0}
       \as_p^{\lambda,\bo}(K) \leq \as_q^{\lambda,\bo}(K)^{1/t} \as_r^{\lambda,\bo}(K)^{1/t'},
    \end{equation*}
    with equality if $K$ is a geodesic ball. Moreover, if $K$ is of class $\cC^{1,1}$ and we have equality, then $K$ is a geodesic ball.
    
    For $r=0$, we have that
    \begin{equation}\label{lambda-monotone_inc}
        \left(\frac{ \as_p^{\lambda,\bo}(K)}{ \as_0^{\lambda,\bo}(K)}\right)^{1+\frac{d}{p}} 
            \leq \left(\frac{ \as_q^{\lambda,\bo}(K)}{ \as_0^{\lambda,\bo}(K)}\right)^{1+\frac{d}{q}},
    \end{equation}
    for all $p,q$ such that $0<p<q$ or $p<q<-d$ or $p>0, q<-d$, with the same equality conditions as above.
    
    Furthermore, for $r\to \infty$, we find that
    \begin{equation}\label{eqn:lambda-monotone}
        \left(\frac{ \as_p^{\lambda,\bo}(K)}{ \as_\infty^{\lambda,\bo}(K)}\right)^{1+\frac{p}{d}} \quad \text{is monotone decreasing for $p\geq 0$.}
    \end{equation}
    
    Finally, for $q\to\infty$ and $r=0$, this yields for $p>0$ that
    \begin{equation}\label{lambda-inequality}
       \as_p^{\lambda,\bo}(K) \leq \as_0^{\lambda,\bo}(K)^{\frac{d}{d+p}} \as_0^{\lambda,\bo}(K^\be)^{\frac{p}{d+p}},
    \end{equation}
    and for $p<0$, $p\neq -d$, the inequality \eqref{lambda-inequality} is reversed.
    Equality holds if and only if $K$ is a geodesic ball.
\end{theorem}
\begin{proof}
The proof  is the same as the one for Theorem  \ref{thm:sphere_inequalities}. For (\ref{lambda-inequality}) we also use Theorem  \ref{Dual}.
\end{proof}

\begin{remark}[Connection between $p$-affine isoperimetric inequality and \eqref{lambda-inequality}]
 From inequality (\ref{lambda-inequality})   we derive  that for $p >0$ and $\lambda =0$, 
    \begin{equation}\label{eqn:p_affine}
        \as_p(\overline{K}) 
        \leq d \vol_d^e(\overline{K})^{\frac{d}{d+p}} \vol_d^e(\overline{K}^\circ)^{\frac{p}{d+p}}
        \leq d \kappa_d^{\frac{2p}{d+p}} \vol_d^e(\overline{K})^{\frac{d-p}{d+p}} = \as_p(B_{\overline{K}}),
    \end{equation}
    where  $B_{\overline{K}}$ is a Euclidean ball with the same (Euclidean) volume as $\overline{K}$ and where in the second inequality we used the Blaschke--Santaló inequality. Thus \eqref{lambda-inequality} can be seen as a real-analytic extension of the $p$-affine isoperimetric inequality $\as_p(\overline{K}) \leq \as_p(B_{\overline{K}})$.
    Note in particular, that for $\lambda=1$, we obtain again
    \begin{equation*}
        \Omega_p^s(K) \leq P^s(K)^{\frac{d}{d+p}} P^s(K^*)^{\frac{p}{d+p}}.
    \end{equation*}
 For $\lambda =1$ and $p=0$ we have by the spherical isoperimetric inequality
    \begin{equation*}
        \as_0^{1, \bo} (K) = \Omega^s_0(K) = P^s(K) \geq P^s(C_K) = \as_0^{1, \bo}(C_K).
    \end{equation*}
    Thus, we cannot expect that a direct analog to \eqref{eqn:p_affine} of the form $\as_p^{\lambda, \bo}(K) \leq \as_p^{\lambda, \bo}(C_K)$ is true for all $p\geq 0$ and $\lambda \geq 0$.
\end{remark}

The monotonicity behavior (\ref{eqn:lambda-monotone}) of Theorem \ref{new-lambda_inequalities} leads us to define an entropy power functional by
\begin{equation*}
    \cE^{\lambda,\bo}(K) 
        := \lim_{q\to 0^+} \left(\frac{ \as_q^{\lambda,\bo}(K^\bo) }{ \as_0^{\lambda,\bo}(K^\bo) } \right)^{1+\frac{d}{q}}.
\end{equation*}
where $K^\bo$ is the $\bo$-polar of $K$, see \cite[Sec.\ 4.4]{BW:2024}. If $K$ is of class $\cC^{1,1}$, then also
\begin{equation*}
    \cE^{\lambda,\bo}(K)=\lim_{p\to\infty} \left(\frac{ \as_p^{\lambda,\bo}(K) }{ \as_\infty^{\lambda,\bo}(K) }\right)^{1+\frac{p}{d}}.
\end{equation*}

By Theorem \ref{new-lambda_inequalities} we have that  
\begin{equation}
    \cE^{\lambda,\bo}(K) \leq   \frac{ \as_0^{\lambda,\bo}(K) }{ \as_0^{\lambda,\bo}(K^\bo) } \leq \frac{\as_0^{\lambda,\bo}(K)}{\as_\infty^{\lambda,\bo}(K)}.
\end{equation}

Analogous to Theorem \ref{thm:KL} we find the following
\begin{theorem}\label{thm:lambda_PW}
Let $\lambda \geq 0$ and let $K \subset  \Sp^d(\lambda)$ be a convex body that contains $\bo\in\Sp^d(\lambda)$ in the interior and is contained in the open half-space centered in $\bo$. Then
\begin{equation*}
    \log \cE^{\lambda,\bo}(K) 
        = \frac{1}{\as_0^{\lambda,\bo}(K^\bo)} 
            \int_{\bd K^\bo} \left[\log \frac{H_{d-1}^\lambda(K^\bo,\bu)}{f_{\bo}^\lambda(K^\bo,\bu)^{d+1}}\right] 
                \, f_{\bo}^\lambda(K,\bu) \, \vol_{\bd K^\bo}^\lambda(\dint \bu),
\end{equation*}
and if $K$ is of class $\cC^{1,1}$, then
\begin{equation*}
    \log  \cE^{\lambda,\bo}(K) 
        = -\frac{1}{ \as_0^{\lambda,\bo}(K^\bo) } \int_{\bd K} \frac{H^\lambda_{d-1}(K,\bu)}{f_{\bo}^\lambda(K,\bu)^{d+1}} 
            \left[\log  \frac{H^\lambda_{d-1}(K,\bu)}{f_{\bo}^\lambda(K,\bu)^{d+1}}\right] \,f_{\bo}^\lambda(K,\bu) \vol_{\bd K}^\lambda(\dint \bu).
\end{equation*}
\end{theorem}

Consequently we define the real-analytic entropy $E_{PW}^{\lambda,\bo}(K)$ by
\begin{equation*}
    E_{PW}^{\lambda,\bo}(K) := -\log  \cE^{\lambda,\bo}(K).
\end{equation*}

We obtain the following  corollary to Theorem \ref{thm:limit_center} and Theorem \ref{thm:lambda_PW}.
\begin{corollary}
    Let $\overline{K}\subset \R^d$ be a convex body that contains the origin in the interior. For $\lambda \geq 0$ and a fixed point $\bo\in \Sp^d(\lambda)$ we consider $\R^d$ as projective model of $\Sp^d(\lambda)$ around $\bo$.
    Then
    \begin{equation*}
        \lim_{\lambda \to 0^+} \cE^{\lambda,\bo}(\overline{K}) = \cE(\overline{K}) \qquad \text{and} \qquad 
        \lim_{\lambda \to 0^+}  E_{PW}^{\lambda,\bo}(\overline{K}) = E_{PW}(\overline{K}).
    \end{equation*}
    Moreover,
    \begin{equation*}
        \lim_{\lambda\to 1} \cE^{\lambda,\bo}(\overline{K}) = \cE^s(\overline{K}) \qquad \text{and} \qquad 
        \lim_{\lambda\to 1} E_{PW}^{\lambda,\bo}(\overline{K}) = E^s(\overline{K}).
    \end{equation*}
\end{corollary}

Thus the real-analytic entropy $E_{PW}^{\lambda,\bo}$ connects the centro-affine entropy $E_{PW}$ in $\R^d$ with the spherical curvature entropy $E^s$ in $\S^d$.

\begin{example}
    For $\lambda>0$ and a  geodesic ball $C^\lambda_{\bo}(\alpha)$ centered in $\bo\in\Sp^d(\lambda)$ of radius $\alpha\in(0,\frac{\pi}{2\sqrt{\lambda}})$ we have
    \begin{equation*}
        \cE^{\lambda,\bo}(C_{\bo}^\lambda(\alpha)) 
            = \frac{\as^{\lambda,\bo}_0(C^\lambda_{\bo}(\alpha))}{\as^{\lambda,\bo}_\infty(C^\lambda_\bo(\alpha))}
            = (\tan_\lambda \alpha)^{d-1} \left(\lambda (\cos_\lambda \alpha)^2 + (\sin_\lambda\alpha)^2\right)^{\frac{d+1}{2}},
    \end{equation*}
    and 
    \begin{equation*}
        E^{\lambda,\bo}_{PW} (C_{\bo}^\lambda(\alpha)) = -(d-1) (\log \tan_\lambda\alpha) - \frac{d+1}{2} \log\left(\lambda (\cos_\lambda \alpha)^2 + (\sin_\lambda\alpha)^2\right).
    \end{equation*}
    We have
    \begin{equation*}
        \lim_{\lambda\to 0^+} E_{PW}^{\lambda,\bo}(C_{\bo}^\lambda(\alpha)) = -2d \log \alpha, \qquad \text{and} \qquad
        \lim_{\lambda\to 1} E_{PW}^{\lambda,\bo}(C_{\bo}^\lambda(\alpha)) = -(d-1)\log \tan\alpha.
    \end{equation*}

\end{example}

\subsection{Real-analytic extension of \texorpdfstring{$E_C$}{Gaussian entropy} to \texorpdfstring{$E^s$}{spherical curvature entropy}}

For a proper  convex body  $K \subset  \Sp^d(\lambda)$ we now consider another real-analytic extension of the spherical $L_p$-floating area $\Omega_p^s$ by
\begin{equation*}
    \Omega^\lambda_p(K) := \int_{\bd K} H_{d-1}^\lambda(K,\bu)^{\frac{p}{d+p}} \, \vol_{\bd K}^\lambda(\dint \bu).
\end{equation*}
Clearly, for $\lambda =1$ we get the $L_p$-floating area $\Omega_p^s$ \eqref{lpfloatarea}, which, as noted above,  also coincides with $\as_p^{1,\bo}$.
Observe also, that for $p=1$ we recover the $\lambda$-floating area  $\Omega^\lambda (K)$ \eqref{eqn:lambda_floating_area_projected}.

\begin{example}
    For a geodesic ball $C^\lambda(\alpha)$ of radius $\alpha \in (0,\frac{\pi}{2\sqrt{\lambda}}]$ we have $H^\lambda_{d-1}(C^\lambda(\alpha),\cdot) \equiv (\tan_\lambda \alpha)^{-(d-1)}$,
    \begin{align*}
        \Omega^\lambda_0(C^\lambda(\alpha)) &= P^\lambda(C^\lambda(\alpha)) = \omega_{d-1} (\sin_\lambda \alpha)^{d-1},\\
        \Omega^\lambda_{\infty}(C^\lambda(\alpha)) & = \omega_{d-1} (\cos_\lambda \alpha)^{d-1} 
            = \lambda^{\frac{d-1}{2}} P^\lambda\left(C^\lambda\left(\frac{\pi}{2\sqrt{\lambda}}-\alpha\right)\right),
    \end{align*}
    and 
    \begin{align*}
        \Omega^\lambda_p(C^\lambda(\alpha)) 
            &= \omega_{d-1} (\sin_\lambda \alpha)^{\frac{d(d-1)}{d+p}} (\cos_\lambda \alpha)^{\frac{p(d-1)}{d+p}}\\
            &=  \lambda^{\frac{(d-1)p}{2(d+p)}}  P^\lambda(C^\lambda(\alpha))^{\frac{d}{d+p}} P^\lambda\left(C^\lambda\left(\frac{\pi}{2\sqrt{\lambda}}-\alpha\right)\right)^{\frac{p}{d+p}}.
    \end{align*}
    Here we recall \eqref{eqn:lambda_sin} for the definition of $\sin_{\lambda}$ and $\cos_\lambda$.
\end{example}

We may use the sphere $\S^d(\lambda) = \frac{1}{\sqrt{\lambda}}\S^{d}\subset \R^{d+1}$ as model for $\Sp^d(\lambda)$, see \eqref{spaceforms}. Then the dual body $K^*\subset \S^d(\lambda)$ of a convex body $K\subset \S^d(\lambda)$ is defined by
\begin{equation*}
    K^* := \frac{1}{\sqrt{\lambda}} (\sqrt{\lambda} K)^*,
\end{equation*}
where $\sqrt{\lambda}K = \{\sqrt{\lambda} \bx : \bx \in K\}\subset \S^d\subset \R^{d+1}$ is a spherical convex body on $\S^d$. The rescaled Gauss map $-\frac{\bn_K}{\sqrt{\lambda}}$ maps $\bd K\subset \S^d(\lambda)$ to $\bd K^*\subset \S^d(\lambda)$. Thus
\begin{equation*}
    H_{d-1}^\lambda(K,\bx) \, H_{d-1}^\lambda\left(K^*, \bx^*\right) = \lambda^{d-1},
\end{equation*}
for almost all $\bx\in\bd K$, where we set $\bx^* := -\frac{\bn_K(\bx)}{\sqrt{\lambda}}\in\bd K^*$.

\begin{example}
    For the geodesic ball $C^\lambda(\alpha)\subset \S^d(\lambda)$ with $\alpha \in (0,\frac{\pi}{2\sqrt{\lambda}})$ we have
    \begin{equation*}
        C^\lambda(\alpha)^* = C^\lambda\left(\overline{\alpha}\right) \qquad \text{where $\overline{\alpha}=\frac{\pi}{2\sqrt{\lambda}}-\alpha$},
    \end{equation*}
    and 
    \begin{equation*}
        H_{d-1}^\lambda(C^\lambda(\alpha),\cdot) \, H_{d-1}^\lambda(C^\lambda(\overline{\alpha}),\cdot) 
            = (\tan_\lambda \alpha)^{-(d-1)} (\tan_\lambda \overline{\alpha})^{-(d-1)} = \lambda^{d-1}.
    \end{equation*}
\end{example}

Now the following duality relation, similar to Theorem \ref{thm:duality} holds.
\begin{theorem}\label{thm:omega_lambda_duality}
    Let $d\geq 2$, $p\in \R\setminus\{-d\}$, $\lambda>0$ and let $K\subset \Sp^d(\lambda)$ be a convex body that is contained in an open half-space. If $p< 0$, then we additionally assume that $K$ is of class $\cC^2_+$ and if $p\in\{0,\infty\}$ then we assume that $K$ is of class $\cC^{1,1}$. Then
    \begin{equation*}
        \Omega_p^\lambda(K^*) = \lambda^{\frac{(d-1)(p-d)}{2(d+p)}} \Omega_{d^2/p}^\lambda(K).
    \end{equation*}
    In particular,
    \begin{equation*}
        \Omega_\infty^\lambda(K) = \lambda^{\frac{d-1}{2}} P^\lambda(K^*),
    \end{equation*}
    if $K$ is of class $\cC^{1,1}$.
\end{theorem}
\begin{proof}
    The proof is analogous to the proof of Theorem \ref{thm:duality} where we note that the Jacobian 
        $J^{\bd K}(\bx^*)(\bx)=\lambda^{-\frac{d-1}{2}} H_{d-1}^\lambda(K,\bx)$.
\end{proof}

\begin{remark}
    One could also use $\Sp^d(1/\lambda)$ to define the dual body $\lambda K^*=\sqrt{\lambda}(\sqrt{\lambda}K)^*\subset \Sp^d(1/\lambda)$, see \cite[Sec.\ 4.3]{BW:2024}. Note that for $\bx\in\bd K\subset \Sp^d(\lambda)$
    \begin{equation*}
        H_{d-1}^{1/\lambda}(\lambda K,\lambda\bx) = \lambda^{-(d-1)} H_{d-1}^{\lambda}(K,\bx).
    \end{equation*}
    Thus
    \begin{equation*}
        \Omega_p^{1/\lambda}(\lambda K)= \lambda^{\frac{d(d-1)}{d+p}} \Omega_p^{\lambda}(K),
    \end{equation*}
    and Theorem \ref{thm:omega_lambda_duality} yields
    \begin{equation*}
        \Omega^{1/\lambda}_p(\lambda K^*) = \lambda^{\frac{d-1}{2}} \Omega_{d^2/p}^{\lambda}(K).
    \end{equation*}
\end{remark}

By fixing a projective model we find the following limit in the Euclidean space.
\begin{theorem}\label{thm:limit_lambda}
    Let $\overline{K}\subset \R^d$ be a convex body. For $\lambda> 0$ we consider $\R^d$ as a projective model of an open half-space of $\Sp^d(\lambda)$. Then
    \begin{equation*}
        \lim_{\lambda\to 0^+} \Omega^\lambda_p(\overline{K}) = \int_{\bd \overline{K}} H_{d-1}^e(\overline{K},\bx)^{\frac{p}{d+p}}\, \vol_{\bd \overline{K}}^e(\dint \bx).
    \end{equation*}
\end{theorem}
\begin{proof}
     By \cite[Eq.~(4.13)]{BW:2018}, we find
     \begin{equation*}
        \Omega_p^\lambda(\overline{K}) 
            =\int_{\bd \overline{K}} H_{d-1}^e(\overline{K},\bx)^{\frac{p}{d+p}} 
                (1+\lambda\|\bx\|_2^2)^{-\frac{d(d-p)}{2(d+p)}} \left(1+\lambda (\bx \cdot \bn_{\overline{K}}(\bx))^2\right)^{\frac{d(1-p)}{2(d+p)}}
                    \, \vol_{\bd \overline{K}}^e(\dint\bx).
    \end{equation*}
    Taking the limit $\lambda\to 0^+$ yields the result.
\end{proof}

\begin{remark}
    We see that for $p=0$ the spherical surface area $P^s(\overline{K}) = \lim_{\lambda\to 1} \Omega^{\lambda}_0(\overline{K})$ can be related to the Euclidean surface area $P(\overline{K}) = \lim_{\lambda\to 0^+} \Omega^{\lambda}_0(\overline{K})$.
\end{remark}

Again, Theorem \ref{thm:sphere_inequalities} naturally extends to $\Omega_p^\lambda$.
  \begin{theorem}\label{thm:lambda_inequalities}
    Let $d\geq 2$, $\lambda> 0$, and $p,q,r\in\mathbb{R}\setminus\{-d\}$. 
    Let $K\subset \Sp^d(\lambda)$ be a convex body that is contained in an open half-space.
    Additionally assume that $K$ is of class $\cC^2_+$ if $\min\{p,q,r\}<0$.
    
    If $t:= \frac{(q-r)(d+p)}{(p-r)(d+q)}>1$ and $t':=\frac{t}{t-1}$, then
    \begin{equation}
        \Omega_p^\lambda(K) \leq \Omega_q^\lambda(K)^{1/t}\, \Omega_r^\lambda(K)^{1/t'}.
    \end{equation}
    with equality if $K$ is a geodesic ball. Moreover, if $K$ is of class $\cC^{1,1}$ and we have equality, then $K$ is a geodesic ball.
    
    For $r=0$, we have that
    \begin{equation}\label{eqn:monotone_inc_omega}
        \left(\frac{\Omega_p^\lambda(K)}{P^\lambda(K)}\right)^{1+\frac{d}{p}} 
            \leq \left(\frac{\Omega_q^\lambda(K)}{P^\lambda(K)}\right)^{1+\frac{d}{q}},
    \end{equation}
    for all $p,q$ such that $0<p<q$ or $p<q<-d$ or $p>0, q<-d$.
    With the same equality conditions as above.
    
    Furthermore, for $r\to \infty$, we find that
    \begin{equation}\label{eqn:lambda-monotone_omega}
        \left(\frac{\Omega_p^\lambda(K)}{\Omega_\infty^\lambda(K)}\right)^{1+\frac{p}{d}} \quad \text{is monotone decreasing for $p\geq 0$.}
    \end{equation}
    
    Finally, for $q\to\infty$ and $r=0$, this yields that for $p>0$,
    \begin{equation}\label{eqn:omega_lambda_ineq}
        \Omega^\lambda_p(K) \leq \lambda^{\frac{(d-1)p}{2(d+p)}} P^\lambda(K)^{\frac{d}{d+p}} P^\lambda(K^*)^{\frac{p}{d+p}},
    \end{equation}
    and for $p<0$, $p\neq -d$, the inequality \eqref{eqn:omega_lambda_ineq} is reversed.
    Equality holds if and only if $K$ is a geodesic ball.
\end{theorem}

Consequently we  define an entropy power functional by
\begin{equation*}
    \cE^\lambda(K) := \lambda^{-(d-1)} \lim_{q\to 0^+}  \left(\frac{\Omega_q^\lambda(K^*)}{P^\lambda(K^*)}\right)^{1+\frac{d}{q}}.
\end{equation*}
If $K$ is of class $\cC^{1,1}$, then by Theorem \ref{thm:omega_lambda_duality}
\begin{equation*}
    \cE^\lambda(K) = \lim_{p\to\infty} \left(\frac{\Omega_p^\lambda(K)}{\Omega_{\infty}^\lambda(K)}\right)^{1+\frac{p}{d}}.
\end{equation*}

By Theorem \ref{thm:lambda_inequalities} we have that 
\begin{equation}
    \cE^\lambda(K) \leq \frac{P^\lambda(K)}{\lambda^{\frac{d-1}{2}}P^\lambda(K^*)}\leq \frac{\Omega_0^\lambda(K)}{\Omega_{\infty}^\lambda(K)}.
\end{equation}

Again, analogous to Theorem \ref{thm:KL} we find the following
\begin{theorem}\label{thm:lambda_Chow}
Let $\lambda >0$ and let $K \subset  \Sp^d(\lambda)$ be a convex body that is contained in an open half-space. 
Then
\begin{equation*}
    \log \cE^{\lambda}(K) = \frac{1}{P^\lambda(K^*)} 
        \int_{\bd K^*} \log H_{d-1}^\lambda(K^*,\bu) \, \vol_{\bd K^*}^\lambda(\dint\bu) - (d-1)\log \lambda.
\end{equation*}
If $K$ is of class $\cC^{1,1}$, then also
\begin{equation*}
    \log  \cE^\lambda(K) 
    = -\frac{1}{\Omega_{\infty}^\lambda(K)} \int_{\bd K} H_{d-1}^\lambda(K,\bu) \left[\log H_{d-1}^\lambda(K,\bu)\right] \, \vol_{\bd K}^\lambda(\dint \bu).
\end{equation*}
\end{theorem}

The proof of this theorem is similar to the proof of Theorem \ref{thm:KL} and we omit it.  The theorem leads us to define the entropy $E_C^\lambda$ by
\begin{equation*}
    E_C^\lambda(K) = -\log  \cE^\lambda(K).
\end{equation*}

We obtain the following  corollary to Theorem \ref{thm:limit_lambda} and Theorem \ref{thm:lambda_Chow}.
\begin{corollary}
    Let $\overline{K}\subset \R^d$ be a convex body. For $\lambda>0$ we consider $\R^d$ as a projective model of an open half-space of $\Sp^d(\lambda)$. Then
    \begin{equation*}
        \lim_{\lambda \to 0^+}  E_C^\lambda(\overline{K}) = E_C(\overline{K}) \qquad \text{and} \qquad 
        \lim_{\lambda \to 1} E_C^\lambda(\overline{K}) = E^s(\overline{K}).
    \end{equation*}
\end{corollary}

Thus the real-analytic Gaussian entropy $E_C^\lambda$ connects the Gaussian entropy $E_C$ in $\R^d$ with the spherical curvature entropy $E^s$ in $\S^d$.

\begin{example}
    For the geodesic ball $C^\lambda(\alpha)\subset \Sp^d(\lambda)$ with $\alpha \in (0,\frac{\pi}{2\sqrt{\lambda}})$ we have
    \begin{equation*}
        \cE^\lambda(C^\lambda(\alpha)) 
        = \frac{\Omega_0^\lambda(C^\lambda(\alpha))}{\Omega_\infty^\lambda(C^\lambda(\alpha))} 
        = (\tan_\lambda\alpha)^{d-1},
    \end{equation*}
    and
    \begin{equation*}
        E_C^\lambda(C^\lambda(\alpha)) = -(d-1) \log \tan_\lambda\alpha.
    \end{equation*}
    Thus
    \begin{equation*}
        \lim_{\lambda\to 0^+} E^\lambda_C(C^\lambda(\alpha)) = -(d-1) \log \alpha \qquad \text{and}\qquad 
        \lim_{\lambda\to 1} E^\lambda_C(C^\lambda(\alpha)) = -(d-1) \log \tan\alpha.
    \end{equation*}
\end{example}

\medskip\noindent
\textbf{Acknowledgments.} Elisabeth Werner was supported by NSF grant \texttt{DMS-2103482}. We extend our sincere gratitude to the Hausdorff Research Institute for Mathematics at the University of Bonn for fostering an outstanding and collaborative working environment during the dual trimester program \textit{Synergies between Modern Probability, Geometric Analysis, and Stochastic Geometry}.

\bibliographystyle{plain}

\medskip
\parindent=0pt

\bigskip
\begin{samepage}
	Florian Besau\\
	Institute of Discrete Mathematics and Geometry\\
	Technische Universität Wien (TU Wien)\\
	Wiedner Hauptstrasse 8--10, 1040 Vienna, Austria\\
	e-mail: florian@besau.xyz
\end{samepage}

\medskip 
\begin{samepage}
	Elisabeth M.\ Werner\\
	Department of Mathematics, Applied Mathematics and Statistics\\
	Case Western Reserve University\\
	10900 Euclid Avenue, Cleveland, Ohio 44106, USA\\
	e-mail: elisabeth.werner@case.edu
\end{samepage}

\end{document}